\documentclass[a4paper,10pt]{amsart}

\usepackage[utf8]{inputenc}
\usepackage[english]{babel}

\usepackage[top=40mm,bottom=35mm,left=39mm,right=39mm]{geometry}
\usepackage{array}
\usepackage{hhline}
\usepackage{caption}
\usepackage{enumitem}
\usepackage{multirow}
\usepackage{xcolor}

\usepackage{amsmath,amsthm,amsfonts,amssymb}

\usepackage{url}
\usepackage[
           colorlinks=true,
           pdfstartview=FitH,
           pdfpagemode=UseNone,
           linkcolor=blue,
           anchorcolor=blue,
           citecolor=blue
           ]{hyperref}

\numberwithin{equation}{section}
\theoremstyle{plain}
\newtheorem{theorem}{Theorem}[section]

\newtheorem{corollary}[theorem]{Corollary}
\newtheorem{problem}[theorem]{Problem}
\newtheorem{conjecture}[theorem]{Conjecture}
\newtheorem{question}[theorem]{Question}

\theoremstyle{definition}
\newtheorem{remark}[theorem]{Remark}
\newtheorem{example}[theorem]{Example}

\def\darkgreen{green!50!darkgray}
\def\darkmagenta{magenta!50!darkgray}

\def\NN{\mathbb{N}}
\def\ZZ{\mathbb{Z}}

\def\QQ{\mathbb{Q}}

\def\twoa{2a}
\def\twob{2b}
\def\twoc{2c}
\def\fivea{5a}
\def\fiveb{5b}

\title[Coprime-Universal Quadratic Forms]{Coprime-Universal Quadratic Forms}

\author{Matteo Bordignon}
\address{
         Charles University,
         Faculty of Mathematics and Physics,
         Department of Algebra,
         Sokolov\-sk\'a 83, 18600 Praha~8,
         Czech Republic
        }
\address{
        Department of Mathematics, KTH, SE-100 44 Stockholm, Sweden
        }
\email{
    bordig@kth.se
    }

\author{Giacomo Cherubini}
\address{
         Charles University,
         Faculty of Mathematics and Physics,
         Department of Algebra,
         Sokolov\-sk\'a 83, 18600 Praha~8,
         Czech Republic
        }
\address{
        Istituto Nazionale di Alta Matematica ``Francesco Severi'',
        Research Unit Dipartimento di Matematica ``Guido Castelnuovo'',
        Sapienza Universit\`a di Roma, Piazzale Aldo Moro 5, I-00185, Roma
        }
\email{
    cherubini@altamatematica.it
    }

\date{\today}


\makeatletter
\newcommand{\addresseshere}{%
  \enddoc@text\let\enddoc@text\relax
}
\makeatother

\begin{document}

\begin{abstract}
Given a prime $p>3$, we characterize positive-definite
integral quadratic forms that are \emph{coprime-universal}
for $p$, i.e.~representing all positive integers coprime to $p$.
This generalizes the $290$-Theorem by Bhargava and Hanke and
extends later works by Rouse ($p=2$) and
De Benedetto and Rouse ($p=3$).
When $p=5,23,29,31$, our results are conditional on the
coprime-universality of specific ternary forms. We prove
this assumption under GRH (for Dirichlet and modular
$L$-functions), following a strategy introduced by Ono and
Soundararajan, together with some more elementary techniques
borrowed from Kaplansky and Bhargava.
Finally, we discuss briefly the problem of representing
all integers in an arithmetic progression.
\end{abstract}

\maketitle

\section{Introduction}

The famous ``$290$-Theorem'', due to Bhargava and Hanke \cite{BH},
states that a positive definite integral quadratic form
represents all positive integers if and only if it represents
a finite set of twenty-nine elements, the largest of which is $290$.

Interest in the representation of integers by quadratic forms
dates back at least to Fermat,
who studied in the 17th century the sum of two squares.
Another well-known example is the sum of four squares, which
represents all positive integers, as was proved by Lagrange in 1770.
Nowadays, we say that a quadratic form with this property is \emph{universal},
so that the 290-theorem characterizes universal quadratic forms.

After Lagrange, several people contributed to the theory of universal quadratic forms.
In 1916, Ramanujan gave a list of diagonal quaternary forms stating they
were universal, a claim later confirmed by Dickson~\cite{Dickson} in 1927
(who corrected a mistake in Ramanujan's list and coined the term `universal').
Willerding~\cite{Willerding} extended the list to \emph{classical} quaternary forms,
namely those with even cross-terms;
and at the end of the 20th century, a complete characterization 
of universal classical quadratic forms was given by Bhargava \cite{Bhargava},
who reproved and greatly simplified an unpublished result by Conway and Schneeberger,
the so-called 15-theorem.

A parallel story developed for quadratic forms representing all positive odd integers
rather than all natural numbers \cite{Jagy,Kaplansky,Oh2,Rouse}, with some surprises.
Going back to Ramanujan, he considered this problem and remarked that the behaviour
of some forms appeared to be unpredictable when the local-global principle fails:
as an example, he wrote \cite[p.14]{Ramanujan} that the odd integers
not represented by $x^2+y^2+10z^2$ \emph{``do not seem to obey any simple law''}.
His example was shown to miss exactly eighteen odd integers by Ono and Soundararajan
in 1997, under the assumption of the Generalized Riemann Hypothesis \cite{OS},
using advanced tools from the theory of modular forms and a numerical computation
to verify the representation of integers up to $10^{10}$.

A complete characterization of quadratic forms representing all positive odd integers
was eventually achieved by Rouse in 2014~\cite{Rouse}, who proved a `451-theorem'
(and there is also an unpublished 33-theorem by Bhargava for classical forms
representing odd integers). A few years later, DeBenedetto and Rouse~\cite{DeBenedetto}
considered quadratic forms representing integers coprime~to~$3$.

In analogy with Bhargava and Hanke's case of
universal forms, we say that a quadratic form
is \emph{coprime-universal with respect to $p$}
(or simply coprime-universal if the prime is understood)
if it represents all positive integers coprime to $p$.
Using this terminology, the works by Rouse and
DeBenedetto--Rouse settle, in two special cases,
the following problem.

\begin{problem}\label{intro:problem}
Given a positive integer $p$, determine the smallest set
$S_p$ with the property that a positive definite integral
quadratic form is coprime-universal with respect to~$p$
if and only if it represents~$S_p$.
\end{problem}

When $p=1$, i.e.~no coprimality is assumed,
$S_1$ is a subset of $\{1,\dots,290\}$ by the $290$-theorem.
When $p=2$, Rouse showed \cite[Theorem~2]{Rouse}
that $S_2$ is a subset of $\{1,3,5,\dots,451\}$,
consisting of $46$ elements.
When $p=3$, we have $S_3\subseteq\{1,\dots,290\}$
and $|S_3|=31$ by \cite[Theorem~1]{DeBenedetto}.
The existence of a set as in Problem \ref{intro:problem},
not completely evident, was proved by Bhargava:
he showed that for any infinite set of positive integers $S$,
there is a finite subset $S_0$ such that integral quadratic forms
represent every integer in $S$ if and only if they represent $S_0$.
For a proof of this (in a more general context), we refer to \cite[Theorem 3.3]{KKO}.
Recently, a version of Bhargava's result over number fields was given in~\cite{C-O}.

In this paper we solve Problem \ref{intro:problem} for every prime $p$,
providing explicit sets $S_p$ in all cases. The feasibility of this follows from
the observation that for $p$ large enough the sets $S_p$ stabilize, hence
leaving us with a finite list of primes to examine.
Similarly to \cite[Theorem~2]{Rouse}, our main theorem is conditional
to the assumption that certain ternary forms are coprime-universal,
which we leave as an open problem and is likely very hard.

\begin{conjecture}\label{intro:conjecture}
The ternary forms listed in Table \ref{table:ternaryforms}
are coprime-universal for the stated primes.
\end{conjecture}

\begin{table}[!ht]
\small
\renewcommand\arraystretch{1.25}
\begin{tabular}{r|>{$}r<{$}|>{$}l<{$}}
\hline
\phantom{xx}$p$ & \text{Conjectured Universal Forms} & \text{Label} \\
\hline
\multirow{3}{*}{2} & x^2 + 2y^2 + 5z^2 + xz          & Q_{\twoa} \\
                   & x^2 + 3y^2 + 6z^2 + xy + 2yz    & Q_{\twob} \\
                   & x^2 + 3y^2 + 7z^2 + xy + xz     & Q_{\twoc} \\
\hline
\multirow{2}{*}{5} & x^2 + xz + 2y^2 + 3yz + 7z^2    & Q_{\fivea} \\
                   & x^2 + 2y^2 + yz + 7z^2          & Q_{\fiveb} \\
\hline
                23 & x^2 + 2xz + 2y^2 + 3yz + 5z^2   & Q_{23} \\
\hline
                29 & x^2 + xz + 2y^2 + 3yz + 5z^2    & Q_{29} \\
\hline
                31 & x^2 + 2xz + 2y^2 + yz + 5z^2    & Q_{31} \\
\hline
\end{tabular}
\captionsetup{size=small,width=76mm}
\caption{Ternary forms conjectured to be relatively universal for the primes $p=2,5,23,29,31$.}\label{table:ternaryforms}
\end{table}

Understanding what integers are represented by ternary forms is
generally a difficult task. The representation numbers $r_Q(n)$,
which count the number of ways of representing
a given integer $n$ by a quadratic form $Q$,
are linked to class numbers of imaginary quadratic fields
and to the value in $s=1$ of quadratic Dirichlet $L$-functions. In turn,
these objects are closely related to the existence of Siegel zeros,
which is a major open problem in number theory.
A different approach goes by exploiting that
the representation numbers can be viewed as Fourier coefficients
of modular forms of weight $3/2$. The theory of modular forms is well understood
when the weight is integral, but less so when the weight is half-integral.
Assuming Conjecture~\ref{intro:conjecture}, our main result is as follows.

\begin{theorem}\label{theo:MAIN}
Let $p$ be a prime. If $p=2,5,23,29$ or $31$, assume Conjecture~\ref{intro:conjecture}.
A positive-definite integral quadratic form is coprime-universal with respect to $p$
if and only if it represents the set $S_p$ appearing in Table~\ref{table:Sp}.
\end{theorem}

This result, aside completing the coprime-universal classification,
gives partial but interesting information about the general situation
with coprimality condition imposed by a composite modulus rather than a prime,
see \eqref{intro:eq:int} for a more precise statement. 
As we mentioned earlier, the cases $p=2,3$ present in
Theorem~\ref{theo:MAIN} were proved in~\cite{DeBenedetto,Rouse}.
We include them in Theorem~\ref{theo:MAIN}
and in the rest of this introduction for the sake of completeness.
A quick inspection of Table~\ref{table:Sp}
shows that the most interesting primes are $p=2,3,5,7$.
In fact, when $p\geq 41$, $S_p$ equals Bhargava and Hanke's set $S_1$,
corresponding to no coprimality assumed.
It is clear why this happens: since no multiple of $41$ belongs to $S_1$,
any quadratic form representing all integers coprime to $41$
must represent all the elements in $S_1$ and therefore must be universal
by the $290$-theorem. The same argument holds for larger primes
and gives Theorem~\ref{theo:MAIN} for $p\geq 41$.
When $2\leq p\leq 37$, the set $S_p$ is obtained from $S_1$ by
adding a handful of new elements and removing all elements divisible by $p$.
If we look at the size of the new exceptions and at the cardinality of $S_p$,
we find that no new element is larger than $451$ and that $|S_p|\leq |S_2|$
for all $p$, which shows that $p=2$ is extremal in both aspects.

\begin{table}[hb]
\small
\renewcommand\arraystretch{1.25}
\begin{tabular}{r|>{$}c<{$}|c}
\hline
\phantom{xx}$p$ & S_p & \\
\hline
\multirow{2}{*}{1} & 1,2,3,5,6,7,10,13,14,15,17,19,21,22,23,26,\phantom{!!!}  & \\
                   & \phantom{!!!} 29,30,31,34,35,37,42,58,93,110,145,203,290 & \\
\hline
\multirow{3}{*}{2} & \phantom{xx}S_1\cup\{11,33,39,41,47,51,53,57,59,77,83,\phantom{xxxxxxxxx} & \\
                   & \phantom{xxxxxx}\!85,87,89,91,105,119,123,133,137,143,\phantom{xxxxx}     & \\
                   & \phantom{xxxxxxxxx}187,195,205,209,231,319,385,451\}-2\ZZ                 & \\
\hline
\multirow{1}{*}{3} & S_1\cup\{11,38,46,47,55,62,70,94,119\} -3\ZZ & \\
\hline
\multirow{1}{*}{5} & S_1\cup\{38,39,46,47,53,61,62,74,78\} -5\ZZ & \\
\hline
\multirow{1}{*}{7} & S_1\cup\{39,46,47,55,62,78,142\} -7\ZZ & \\
\hline
$\geq 11$ & S_1 - p\ZZ &\\
\hline
\end{tabular}
\captionsetup{size=small,width=90mm}
\caption{Sets $S_1$ and $S_p$ for each prime.}\label{table:Sp}
\end{table}

Loosely speaking, the proof of Theorem \ref{theo:MAIN}
starts by constructing a finite list of ternary and quaternary forms,
called escalators (see Section~\ref{sec:coprime}),
which are likely to be coprime-universal.
Then, we show that any coprime-universal quadratic form
must represent one of the escalators in our list.
In light of Conjecture~\ref{intro:conjecture},
it is interesting to understand what happens with ternary forms.
Regarding this, we have a complete classification.

\begin{theorem}
Let $p$ be a prime.
There exist no positive-definite integral ternary quadratic forms
that are coprime-universal when $p=19$ or $p\geq 37$.
Instead, they exist (unconditionally) when $p=2,3,5,7,11,13,17$,
and conditionally on Conjecture~\ref{intro:conjecture} when $p=23,29,31$.
\end{theorem}

A more precise version of the above theorem is given in Theorem \ref{S3:thm},
where we show that there is exactly one possible coprime-universal ternary
quadratic form when $p=11,13,17,23,29,31$, while there are
two forms when $p=7$ and a total of eight forms when $p=5$.
When $p=19$ or $p\geq 37$, the situation is similar to the case $p=1$,
since all ternary escalators have exceptions.

We also solve Problem \ref{intro:problem} for classical forms.
Since none of the forms in Table \ref{table:ternaryforms} is classical,
one may hope to have an unconditional result.
We show that this is indeed the case.

\begin{theorem}\label{theo:MAIN1}
Let $p$ be a prime. A positive-definite classical quadratic form
is coprime-universal with respect to $p$ if and only if it
represents the set $S_p'$ in Table~\ref{table:SpPrime}.
\end{theorem}

\begin{table}[htb]
\small
\renewcommand\arraystretch{1.25}
\begin{tabular}{r|>{$}c<{$}|c}
\hline
\phantom{xx}$p$ &         S_p'                    &\\
\hline
  1     & 1, 2, 3, 5, 6, 7, 10, 14, 15                  &\\
\hline
  2     & 1, 3, 5, 7, 11, 15, 33                        &\\
\hline
  3     & 1, 2, 5, 7, 10, 11, 14, 19, 22, 31, 35        &\\
\hline
  5     & 1, 2, 3, 6, 7, 13, 14, 21, 26, 29, 58         &\\
\hline
  7     & 1, 2, 3, 5, 6, 10, 15, 23, 30, 31, 39, 55, 78 &\\
\hline
$\geq 11$ & S_1' &\\
\hline
\end{tabular}
\captionsetup{size=small,width=90mm}
\caption{Sets $S_1'$ and $S_p'$ for each prime.}\label{table:SpPrime}
\end{table}

Arguing as with Theorem \ref{theo:MAIN}, one shows that
the case $p\geq 11$ of Theorem \ref{theo:MAIN1}
follows from the fact that no multiple of $p$ belongs to
the set $S_1'$ in the 15-Theorem for classical forms
without coprimality assumption~\cite[Theorem~1]{Bhargava}.
On the other hand, unlike Theorem~\ref{theo:MAIN} where
the prime $p=2$ was extremal, we find here that the extremal case
for classical forms occurs when $p=7$,
both in terms of the largest exception and of the cardinality of $S_p'$.

The sets $S_p$ and $S_p'$ appearing in
Theorem~\ref{theo:MAIN} and Theorem~\ref{theo:MAIN1}
are minimal, as required by Problem \ref{intro:problem}.
In other words, no element $t\in S_p$ (resp. $t'\in S_p'$)
can be removed from the list, because
there are quadratic forms missing it.
The largest element $t_{\max}$ in each set is also a little special,
because not only there are quadratic forms missing it,
but in fact any quadratic form representing everything below $t_{\max}$
represents everything above $t_{\max}$.
We express these results in the following two corollaries.

\begin{corollary}\label{intro:corollary}
For each prime $p$, let $S_p$ (resp.~$S_p'$) be as in Table~\ref{table:Sp} (resp.~Table~\ref{table:SpPrime}).
For every $t\in S_p$ (resp.~$t'\in S_p'$) there is a positive-definite integral (resp.~classical)
quadratic form representing every integer coprime to $p$ except $t$ (resp.~$t'$).
\end{corollary}

\begin{corollary}\label{intro:cor2}
Let $p$ be a prime and let $Q$ be a positive-definite integral quadratic~form.

$i)$ Assume that $Q$ represents every positive integer coprime to $p$ smaller than $t_{\max}=\max(S_p)$.
When $p\in\{2,5,23,29,31\}$, assume Conjecture~\ref{intro:conjecture}.
Then $Q$ represents every integer coprime to $p$ greater than $t_{\max}$.

$ii)$ Assume that $Q$ is classical and represents every positive integer coprime to $p$ smaller than $t_{\max}'=\max(S_p')$.
Then, unconditionally for all primes,
$Q$ represents every integer coprime to $p$ greater than $t_{\max}'$.
\end{corollary}

In Table~\ref{table:forms4truants} in the Appendix we provide a list of quadratic forms
representing everything below~$t$ for each element $t\in S_p$, for the various
primes $p$. This is used in Section~\ref{S8} to produce quadratic forms with
exactly one exception as in Corollary \ref{intro:corollary}.
The proof of Corollary \ref{intro:cor2} is obtained by inspecting carefully
the list of escalators used to prove Theorem~\ref{theo:MAIN}.

To analyze quaternary escalators and verify whether they are coprime-universal,
we use a combination of techniques from \cite{BH,DeBenedetto,Rouse}, adapted to suit our needs.
The first two of these (see~\S\ref{S4.2}--\S\ref{S4.3})
require checking if a quaternary form $Q$ represents
a coprime-universal ternary form (when they exist), in which case we are done with~$Q$;
or, in alternative, if a regular ternary form, say $T$, is represented.
When this happens, we look at the local obstructions of $T$
and show that an orthogonal complement of $T$ in $Q$ fixes these obstructions,
so that $Q$ misses a finite explicit list of integers.

The third and fourth methods (Section~\ref{S5}) have a more analytic flavour
and use modular forms, starting from the observation that
the theta series $\theta_Q$ associated to $Q$ is a modular form of weight 2.
As such, it decomposes into an Eisenstein part $E$ and a cuspidal part $C$:
\begin{equation}\label{intro:eq001}
\theta_Q(z) = \sum_{n\geq 0} r_Q(n)q^n = \sum_{n\geq 0} a_E(n)q^n + \sum_{n\geq 1}a_C(n)q^n\qquad q=e^{2\pi iz}.
\end{equation}
The Eisenstein coefficient $a_E(n)$ grows linearly in $n$ and so we have $a_E(n)>C_1n$;
on the other hand, by Deligne's theorem, we have $|a_C(n)|\leq C_2d(n)\sqrt{n}$.
Assuming $r_Q(n)=0$ leads to a contradiction for $n$ sufficiently large,
which leaves us with checking the representation of a finite list of integers.
The constant $C_1$ is computed exactly,
while the constant $C_2$ is either computed exactly (method 4, see \S\ref{S5.2})
by doing extensive linear algebra on modular form spaces
(possibly of very large dimension), or is approximated
(method~3, see \S\ref{S5.1}) by using explicit bounds proved by Rouse \cite{Rouse}.

We also include an initial test with the 290-theorem, which we call `Method zero'.
This shows that in fact a good proportion of our quaternary escalators represents
every integer up to 290 and is therefore universal, see \S\ref{S4.1}.

The proof of Theorem~\ref{theo:MAIN} -- Corollary \ref{intro:cor2}
is then given in Sections \ref{S6}--\ref{S8}, where the key point is
that any coprime-universal quadratic form must represent one of the escalators, so that
the set $S_p$ is obtained by looking at all the exceptions found in the course
of the escalation process.

In Sections \ref{sec:elementary}--\ref{sec:GRH} we address Conjecture \ref{intro:conjecture}.
A standard method to deal with ternary quadratic forms consists in showing that they are regular,
which can be done by showing that they are unique in their genus, see e.g.~\cite{J-P}.
When this method fails, one can sometimes apply better tailored techniques, as was done by Kaplansky \cite[\S 4]{Kaplansky}
and Jagy \cite{Jagy} to prove that certain ternary forms represent all odd integers.
These techniques rely on the existence of auxiliary regular forms
from which we can go back to the original one by means of an elementary change of variable;
see Section \ref{sec:elementary} for more details.

A second method uses embedded genera: we show that a given quadratic form $Q$
contains the representatives of a whole genus. It follows that $Q$ represents
everything that is represented by such a genus, which can be easily checked by local conditions.
This method was already used by Bhargava \cite[p.~30]{Bhargava}
and typically does not resolve all the local conditions,
but is nevertheless useful for some of them.
With these arguments, we make partial progress towards Conjecture~\ref{intro:conjecture}.

\begin{theorem}\label{theo:partial1}
Consider the forms $Q_{\twoa},\dots,Q_{31}$ appearing in Table~\ref{table:ternaryforms}.
For given $p\in\{2,5,23,29,31\}$, let $N$ denote a squarefree positive integer
coprime with $p$. If $N$ satisfies any of the congruences listed
in Table~\ref{table:elementary}, then $N$ is represented by the form
indicated in the table.
\end{theorem}

\begin{table}[ht]
\small
\renewcommand\arraystretch{1.25}
\begin{tabular}{r|>{$}c<{$}|>{$}c<{$}}
\hline
$p$ & \text{\rm{Represented congruences}} & \text{\rm{Form}}\\
\hline
\multirow{3}{*}{$2$} & N\equiv 0\pmod{19} &   Q_{\twoa}\\
\hhline{~|--}
                     & N\equiv 0\pmod{31} &   Q_{\twob}\\
\hhline{~|--}
                     & N\equiv 0\pmod{37} &   Q_{\twoc}\\
\hline
\multirow{2}{*}{$5$} & N\equiv 0,2\pmod{3}\text{\rm{ or }} N\equiv 3\pmod{4} & Q_{\fivea}\\
\hhline{~|--}
                     & N\equiv 0\pmod{11} \text{\rm{ or }} N\equiv 1\pmod{4} & Q_{\fiveb}\\
\hline
                $23$ & N\equiv 1\pmod{4}  &   Q_{23}\\
\hline
                $29$ & N\equiv 3\pmod{8}  &   Q_{29}\\
\hline
                $31$ & N\equiv 1\pmod{4}  &   Q_{31}\\
\hline
\end{tabular}
\captionsetup{size=small,width=90mm}
\caption{If $N$ is a squarefree positive integer coprime to~$p$
and satisfies one of the stated congruences, then $N$ is represented
by the corresponding form on the right.}\label{table:elementary}
\end{table}

In fact, by the method of embedded genera we find more local conditions
than the ones listed in Table \ref{table:elementary};
we do not include them in Theorem \ref{theo:partial1} because they involve
primes not dividing the discriminant of our forms. As an example,
we find that $Q_{29}$ contains a genus of size 13 representing
all squarefree integers $n$ with $29\nmid n$ and $n\equiv 1\pmod{3}$.
For further details on these additional congruences, see Remark~\ref{sec:elementary:rmk}.

Theorem~\ref{theo:partial1} is new even when $p=2$: on~\cite[p.1741]{Rouse}
it is only mentioned that, for $Q_{\twoa}$, if $N=19N_0$ with $N_0$ not a quadratic residue modulo $19$,
then $N$ is represented. We prove that it suffices to have the divisibility by 19
to get to the same conclusion and we further extend the result to $Q_{\twob}$ and $Q_{\twoc}$.

In Section \ref{sec:GRH} we give a full resolution
of Conjecture \ref{intro:conjecture} under the assumption
of the Generalized Riemann Hypothesis.

\begin{theorem}\label{thm:GRH}
The Generalized Riemann Hypothesis implies Conjecture \ref{intro:conjecture}.
\end{theorem}

Theorem \ref{thm:GRH} should be compared with \cite[Theorem 7]{Rouse},
where the case $p=2$ was discussed, and with \cite[Theorem 1.1]{LO}
about regular ternary forms. Also, to be precise,
we need GRH for specific quadratic Dirichlet $L$-functions
and modular $L$-functions of weight two.
The proof of Theorem \ref{thm:GRH} follows a path
pioneered by Ono--Soundararajan to study Ramanujan's ternary form \cite{OS}
and begins with a decomposition of theta series as in \eqref{intro:eq001}.
Since we work with ternary forms, the Eisenstein coefficient $a_E(n)$
equals roughly $\sqrt{n}L(1,\chi)$, where $\chi$ is a quadratic character of modulus
a multiple of $n$, while the cuspidal part is understood as the central value
of a weight-two modular $L$-function by means of the Shimura lift and Waldspurger's formula.
Explicit upper and lower bounds on the $L$-functions are obtained
under the assumption of GRH, leading to an inequality of the type
$a_E(n)>|a_C(n)|$ for $n$ sufficiently large, which leaves us with
a finite list of integers to check.
For more details, see \S\ref{sec:GRH.1}.

In the proof of Theorem~\ref{thm:GRH} we also take advantage of the partial results
provided by Theorem~\ref{theo:partial1}, since we can restrict in most cases to
primitive characters $\chi$, which simplifies the numerics,
see~\S\ref{sec:GRH.3} and~\eqref{1905:eq001}.

In Section \ref{sec:stat} we discuss possible extensions of our results to
quadratic forms representing an arithmetic progression rather than all the integers coprime to
a given prime. Although the overall strategy of proof should in principle go through,
we find that some parts of it could be computationally quite challenging.
To avoid doing an exhaustive list of escalators, in this case
we build escalators `at random' and try to determine what kind of
exceptions we get by repeating several of these tosses.
Our data is limited, so we do not claim any results or conjectures,
but we state an open question suggested by the numerics.
We direct the reader to Section \ref{sec:stat} for more details.

\begin{question}\label{intro:question}
Let $m$ be a positive integer and $0\leq k\leq m-1$.
Let $S_{mk}\subseteq m\NN+k$ be a minimal set with the property that
if a positive-definite integral quadratic form represents $S_{mk}$
then it represents all of $m\NN+k$.
Do we have $\max(S_{mk})=O(m)$?
\end{question}

Another natural generalization of our results would be to integers coprime
to a composite modulus $q$ rather than a prime. In order to have
$S_q\neq S_1$, there must be a prime $p$ dividing $q$ with $p\in\{2,\dots,37\}$.
Moreover, we have
\begin{equation}\label{intro:eq:int}
S_q \supseteq \bigcap_{p|q} S_p.
\end{equation}
Indeed, if $m\in S_p$ for all $p|q$, then we clearly have $(m,q)=1$.
Next, for any prime $p|q$, there is a form $Q$ representing all integers coprime to $p$ less than $m$.
But then $Q$ is an integral quadratic form representing all integers coprime to $q$ and missing $m$,
hence $m\in S_q$, proving \eqref{intro:eq:int}.
In general, we expect the converse of \eqref{intro:eq:int} to fail;
nevertheless, we can still use the inclusion to gain some insight on $S_q$,
as we may have a large intersection (take for example $q=17\times 23$, when the intersection
equals $S_1$ minus just two elements).

A directly related field of research concerns
universal quadratic forms over number fields.
For the precise definition and recent literature in this area,
we refer to \cite{kalasurvey,Kim}. Some remarkable phenomena can occur:
Maass~\cite{maass} showed that $x^2+y^2+z^2$
is universal over $\mathbb{Q}(\sqrt{5})$,
in contrast with the rational case where ternary forms
alwasy have local obstructions \cite{doylewilliams}.
In general, universal quadratic forms tend to have many
variables \cite{blomerkala,kalatinkova,KYZ}
and a major open problem is Kitaoka's conjecture,
which states that there are
only finitely many number fields admitting a universal ternary form.
Partial results towards this are given in \cite{KY}.
We are not aware of 290-theorems over any number field,
although there are a 15-theorem over $\QQ(\sqrt{5})$, see~\cite{lee},
as well as other partial results over quadratic fields \cite{deutsch1,deutsch2,sasaki}.

\subsection*{Acnokwledgements}
We would like to thank Vitezslav Kala for introducting us to the topic and for several suggestions.
We also thank Pavlo Yatsyna for fruitful discussions.
We thank Faruk G\"olo\u{g}lu and Pietro Mercuri for their help with running \textsc{magma}.
The authors would like to thank the Institut Mittag-Leffler for the hospitality during the thematic semester on Analytic Number Theory in 2024.

\subsection*{Funding}
M.B.~was supported by OP RDE project No.~\path{CZ.02.2.69/0.0/0.0/18_053/0016976}
International mobility of research, technical and administrative staff at the Charles University
and was partially supported by the Swedish Research Council (2020-04036).
G.C.~was supported by Czech Science Foundation GACR, grant~\texttt{21-00420M}.

\section{Quadratic forms and lattices}\label{S2}

Let $Q$ be a quadratic form in $r$ variables, so that
$Q(x)=\frac{1}{2}x^t Ax$ for some $r\times r$ symmetric matrix $A$.
The form $Q$ is \emph{integral} (or `integer-valued')
if $A$ has integer coefficients and even diagonal;
$Q$ is \emph{classical}
(or `integer-matrix')
if all the coefficients of $A$ are even.

We associate a lattice $L$ to a positive definite integral form $Q$
by letting $L=\ZZ^r$ and defining an inner product on $L$ by setting
\begin{equation}\label{innerproduct}
\langle x,y\rangle = \frac{1}{2}(Q(x+y)-Q(x)-Q(y)).
\end{equation}
Clearly, $\langle x,x\rangle=Q(x)$ is integral,
but arbitrary products $\langle x,y\rangle$ need not be integral.
In matrix form, \eqref{innerproduct} is given by $\frac{1}{2}x^tAy$,
so that $A$ is the Gram matrix of $L$.
Throughout the paper, we will often use lattice-related terminology for quadratic forms
and vice versa.

Let $n$ be a positive integer. We say that $Q$ \emph{represents} $n$ if there is
$x\in\ZZ^r$ such that $Q(x)=n$. If $n$ is not represented by $Q$, we say that
$n$ is an \emph{exception} for $Q$. The smallest exception of a form is called its \emph{truant}.
Similarly, let $p$ be a prime and let $\ZZ_p$ denote the ring of $p$-adic integers.
We say that $Q$ represents $n$ locally if for every $p$ there is some $x\in\ZZ_p^r$
such that $Q(x)=m$. A form is \emph{regular} if it represents
every integer that is represented locally.

Finally, if $Q'$ is a quadratic form of rank $r'\leq r$, we say that $Q$ represents $Q'$
if there is a $r\times r'$ matrix $M$ with integer coefficients such that $Q(My)=Q'(y)$
for all $y\in\ZZ^{r'}$. Note that this happens if and only if there is a sublattice
$L''\subseteq L$ isometric to $L'$.

\section{Binary, ternary and quaternary escalators}\label{sec:coprime}

Let $p$ be a prime and write $T_p=\{t\geq 1: (t,p)=1\}=\{t_1,t_2,\dots\}$.
We say that a positive definite quadratic form $Q$ (resp.~a lattice $L$) is
coprime-universal with respect to $p$ if it represents every integer in $T_p$
(resp.~it contains a vector of norm $t$ for every $t\in T_p$).
When the prime $p$ is clear from the context, we will simply say
that $Q$ is coprime-universal.

Given a quadratic form $Q$ with associated lattice $L$,
let $t$ be an exception for~$Q$. An \emph{escalation} of $L$
(or `escalator lattice', or simply `escalator')
is a lattice $L'$ generated by $L$ and a vector of norm $t$.

If we start with a coprime-universal lattice $L$,
we can construct a chain of escalator sublattices in $L$ as follows:
since $L$ is coprime-universal, it certainly will contain a
one-dimensional lattice $L_1$ generated by a vector of norm $t_1$.
We can escalate $L_1$ by the smallest $t_j$ not represented (the truant of $L_1$)
and obtain a rank-two sublattice $L_2$ in $L$, which we can escalate
by the truant of $L_2$, and so on.
This construction produces a sequence of escalator lattices
\[
\{0\}=L_0\subseteq L_1\subseteq L_2\subseteq L_3\subseteq\cdots \subseteq L.
\]
Since $L\cong \ZZ^r$ is a Noetherian $\ZZ$-module, this sequence
eventually stabilizes. Thus, there is some $L_n\subseteq L$ that is
coprime-universal.
As it will emerge from our proof, it suffices to
stop at $n=4$ in the majority of cases.

We begin by escalating the zero-dimensional matrix,
which gives us the one-dimensional lattice with Gram matrix
$A=(2)$, corresponding to the quadratic form $Q(x)=x^2$.
This has truant $2$, so its escalations have Gram matrix
\[
\begin{pmatrix}
2 & a\\
a & 4
\end{pmatrix}
\]
with $a\in\ZZ$ and $a^2\leq 8$ by Cauchy--Schwarz.
Up to isometry, we can assume $a\geq 0$.
Thus, we obtain three non-isometric escalators of rank two:
\begin{equation}\label{eq:22}
A_1 =
\begin{pmatrix}
2 & 0 \\
0 & 4 
\end{pmatrix},
\quad
A_2 =
\begin{pmatrix}
2 & 1 \\
1 & 4 
\end{pmatrix},
\quad
A_3 =
\begin{pmatrix}
2 & 2 \\
2 & 4 
\end{pmatrix}.
\end{equation}
Note that $A_3$ is equivalent to the diagonal matrix $2\mathbf{I}_2$.
The associated quadratic forms $Q_1(x,y)=x^2+2y^2$,
$Q_2(x,y)=x^2+xy+2y^2$ and $Q_3(x,y)=x^2+2xy+2y^2$
(or, equivalently, $Q_3(x,y)=x^2+y^2$) have exceptions
\[
\begin{split}
\mathrm{Exceptions}(Q_1) &= 5, 7, 10, \dots\\
\mathrm{Exceptions}(Q_2) &= 3, 5, 6, \dots\\
\mathrm{Exceptions}(Q_3) &= 3, 6, 7, \dots\\
\end{split}
\]
Therefore, when $p=5$ we can
escalate $A_1,A_2$ and $A_3$ by $7,3$ and $3$,
respectively. When $p\geq 7$ we use $5,3$ and $3$ instead.
This way we obtain 43 non-isometric
escalators of rank three when $p=5$
and 35 escalators when $p\geq 7$,
see Table~\ref{table:ternary-summary}.

Some of the ternary escalators have no exceptions
coprime to $p$ less than 1000 and 
in fact we can prove they are coprime-universal.

\begin{table}[!ht]
\small
\renewcommand\arraystretch{1.25}
\begin{tabular}{r|>{$}r<{$}|l}
\hline
           $p$     & \text{Coprime Universal Forms}  & \\
\hline
\multirow{6}{*}{5} & x^2 +5xz + 2y^2  +yz +7z^2      & \\
                   & x^2 +4xz + 2y^2 +3yz +7z^2      & \\
                   & x^2 +4xz + 2y^2 +2yz +7z^2      & \\
                   & x^2 +2y^2 +4yz +7z^2            & \\
                   & x^2 +xy +2xz +2y^2 +3yz +3z^2   &\\
                   & x^2 +xy +2y^2 +yz +3z^2         &\\
\hline
\multirow{2}{*}{7} & x^2 +4xz +2y^2 +yz +5z^2        & \\
                   & x^2 +3xz +2y^2 +yz +5z^2        & \\
\hline
                11 & x^2 +xz +y^2 +3z^2              & \\
\hline
                13 & x^2 +3xz +2y^2 +3yz +5z^2       & \\
\hline
                17 & x^2 +xy +xz +2y^2 +2yz +3z^2    & \\
\hline
\end{tabular}
\captionsetup{size=small,width=76mm}
\caption{Provably coprime-universal ternary forms.}\label{provablyternary}
\end{table}

\begin{theorem}\label{S3:thm}
Assume $5\leq p\leq 17$.
The ternary forms in Table \ref{provablyternary} are coprime-universal.
If $p=19$ or $37$, there is no coprime-universal ternary quadratic~form.
\end{theorem}

\begin{proof}
Assume $5\leq p\leq 17$. Each form in Table \ref{provablyternary}
locally represents everything coprime to $p$.
Moreover, they are all regular, as one can readily verify
by checking the list of provably regular forms in \cite{913}.
Hence, they represent all integers coprime to $p$.
When $p=19$ or $37$, all the ternary escalators
have at least one exception coprime to $p$.
\end{proof}

In addition, when $p=5,23,29$ and $31$,
there are also ternary forms that do not have exceptions
up to 1000, but for which we cannot prove coprime-universality,
see Table~\ref{table:ternaryforms}.
They represent everything coprime to $p$ locally,
but are not in the list \cite{913} of regular ternary forms.
Conjecture \ref{intro:conjecture} states
that these forms are indeed coprime-universal, and we assume
from now on the validity of this conjecture.

Summarizing, the situation with ternary forms is described
in Table~\ref{table:ternary-summary}.

\begin{table}[!ht]
\begin{tabular}{>{\raggedleft\footnotesize\hspace{0pt}}r*{4}{|>{\raggedleft\footnotesize\hspace{0pt}}p{18mm}}|}
\hline
$p$ & Ternary forms & Provably coprime-universal & Conjectured coprime-universal & With exceptions \tabularnewline
\hline
         1  &  34  &   0  &  0  &  34 \tabularnewline
\hline
         2  &  73  &  20  &  3  &  50 \tabularnewline
\hline
         3  &  50  &  11  &  0  &  39 \tabularnewline
\hline
         5  &  43  &   6  &  2  &  35 \tabularnewline
\hline
         7  &  35  &   2  &  0  &  33 \tabularnewline
\hline
11, 13, 17  &  35  &   1  &  0  &  34 \tabularnewline
\hline
23, 29, 31  &  35  &   0  &  1  &  34 \tabularnewline
\hline
    19, 37  &  35  &   0  &  0  &  35 \tabularnewline
\hline
\end{tabular}
\captionsetup{size=footnotesize,width=102mm}
\caption{Number of ternary escalators. When $p\geq 7$, the total set of ternary escalators is the same, although different forms are coprime-universal (conjecturally or provably) with respect to different primes, see Table~\ref{table:ternaryforms}
and Table~\ref{provablyternary}.
}\label{table:ternary-summary}
\end{table}

\subsection{Quaternary escalators}\label{S3.1}
Next we escalate the ternary forms having at least one exception.
There are $35$ such forms when $p\in\{5,19,37\}$,
$33$ when $p=7$ and $34$ when $p\in\{11,13,17,23,29,31\}$.

This step produces a large set of matrices of rank four
(between five and ten thousand for each prime approximately),
as described in Table \ref{table:quat}.
Most matrices represent everything coprime to $p$ locally
(hence, in particular, we can start checking if they are coprime-universal globally).
We will refer to these forms as the `basic quaternary forms'.
Yet, a handful of escalators still have local obstructions
(hence, they cannot be coprime-universal). An example is the form
\begin{equation}\label{2306:eq001}
Q(x,y,z,w) = x^2 + 2y^2 + 7z^2 + 7zw + 14w^2,
\end{equation}
which appears among the four-dimensional escalators
when $p=5$, but fails to represent all integers of the form $7^ra$
with $r$ odd and $a\equiv 3,5$ or $6\pmod{7}$.

To handle these `critical quaternary forms',
instead of escalating again to rank five,
we argue as in \cite[\S 5.2]{BH} and \cite{DeBenedetto, Rouse}:
we recover the ternary sublattice from which the forms were obtained,
and escalate it with a number other than the truant.

For instance, the form in \eqref{2306:eq001} is obtained by
adjoining a vector of norm $14$ to the ternary
sublattice $x^2+2y^2+7z^2$, whose exceptions~$\leq 50$ are
$5,14,20,21,35$ and~$42$;
note that $Q$, too, misses $5,21,35$ and~$42$.
There are $22$ critical quaternary forms when $p=5$;
aside from $Q$, one more, namely
\(
x^2 + 2y^2 + 7yw + 7z^2 + 7zw + 14w^2,
\)
comes from the same ternary sublattice and again it fails to represent $21$.
Therefore, any five-dimensional escalator coming from either of
these two forms is obtained by adjoining a vector of norm~$21$.

This leads to the idea of the `switch': escalating first by~$14$
and then by $21$ can be done in reverse order.
Thus, we consider a new set of `auxiliary quaternary escalators'
obtained from $x^2+2y^2+7z^2$ by adjoining a vector of norm~$21$.
After discarding any lattice isometric to a previously found one,
this produces a new set of quaternary escalators,
which fortunately have no local obstructions.

For the remaining $20$ critical quaternary forms we argue similarly: 
six of them come from a ternary sublattice to which we apply a 13--14 switch,
and for all the others, which come from two distinct ternary sublattices
(cf. Table \ref{table:ternaryfortheswitch}), we do a 14--26 switch.
We also note that, for the escalators of $x^2+xz+2y^2+7z^2$,
their common truant $26$ is the third exception, and not the second one,
of the ternary sublattice.
Altogether, we originally obtain 10451 quaternary escalators,
of which 10429 basic, to which we add 1188 auxiliary ones.
Understanding the squarefree integers represented by these $10429+1188=11617$
quadratic forms will suffice to prove Theorem \ref{theo:MAIN} when $p=5$.

\begin{table}[!ht]
\begin{tabular}{>{\raggedleft\footnotesize\hspace{0pt}}c*{5}{|>{\raggedleft\footnotesize\hspace{0pt}}p{17mm}}|}
\hline
$p$ & Escalators & Basic & Local obstructions & Original Ternaries & Auxiliary \tabularnewline
\hline
 1 &  6560 &  6555 &  5 & 1 &  104 \tabularnewline
\hline
 2 & 24312 & 24308 &  4 & 2 &  580 \tabularnewline
\hline
 3 &  8894 &  8884 & 10 & 3 &  727 \tabularnewline
\hline
 5 & 10451 & 10429 & 22 & 4 & 1188 \tabularnewline
\hline
 7 &  7563 &  7558 &  5 & 1 &  486 \tabularnewline
\hline
11 &  6344 &  6339 &  5 & 1 &  104 \tabularnewline
\hline
13 &  6849 &  6844 &  5 & 1 &  104 \tabularnewline
\hline
17 &  6320 &  6315 &  5 & 1 &  104 \tabularnewline
\hline
19 &  6572 &  6567 &  5 & 1 &  104 \tabularnewline
\hline
23 &  6144 &  6139 &  5 & 1 &  104 \tabularnewline
\hline
29 &  5898 &  5893 &  5 & 1 &  104 \tabularnewline
\hline
31 &  5750 &  5745 &  5 & 1 &  104 \tabularnewline
\hline
37 &  6572 &  6567 &  5 & 1 &  104  \tabularnewline
\hline
\end{tabular}
\captionsetup{size=footnotesize,width=110mm}
\caption{Number of quaternary escalators, basic forms and auxiliary forms.}\label{table:quat}
\end{table}

Regarding the primes $7\leq p\leq 37$, the situation is slightly simpler.
Among the quaternary escalators, only five have local obstructions (always at the prime $2$):
\begin{equation}\label{2808:eq001}
\begin{split}
&x^2+xz+2y^2+2yz+5z^2+2xw-2zw+10w^2\\
&x^2+xz+2y^2+2yz+5z^2+2xw-4zw+10w^2\\
&x^2+xz+2y^2+2yz+5z^2+4xw+4yw+2zw+10w^2\\
&x^2+xz+2y^2+2yz+5z^2+4xw+4yw-2zw+10w^2\\
&x^2+xz+2y^2+2yz+5z^2+6xw+2zw+10w^2\\
\end{split}
\end{equation}
All of these forms have $14$ and $30$ as their first two exceptions
and are obtained by adjoining a vector of norm $10$ to the same ternary sublattice,
namely $x^2+xz+2y^2+2yz+5z^2$, which has exceptions $10,13,14,30,40$ and $46$ up to 50.
When $p=7$ we do a 10--30 switch, obtaining 486 auxiliary escalators, whereas
when $11\leq p\leq 37$ we do a 10--14 switch, which produces just 104 auxiliary lattices.
In fact, the forms in \eqref{2808:eq001} are the same appearing in Bhargava and Hanke's escalation process
when $p=1$ and our 10--14 switch corresponds to the one they did in \cite[\S 5.2]{BH}.

\begin{table}[!ht]
\small
\renewcommand\arraystretch{1.25}
\begin{tabular}{r|>{$}r<{$}|r|>{\centering\hspace{0pt}}p{15mm}|}
\hline
$p$ & \text{Ternary Sublattices\phantom{\dots}} & Exceptions\phantom{ \dots} & Used for the switch\tabularnewline
\hline
\multirow{4}{*}{5}
   & x^2 -3xz +2y^2 -2yz +7z^2 & 10,13,14,30,40 \dots& 14 \tabularnewline
   & x^2 +2y^2 +7z^2 & \phantom{x} 5,14,20,21,35,42 \dots& 21
   \tabularnewline
   & x^2 -3xz +2y^2 +7z^2 & 10,14,26,40,42 \dots& 26 \tabularnewline
   & x^2 -xz +2y^2 +7z^2 & \phantom{x} 5,10,14,20,23,26 \dots& 26 \tabularnewline
\hline
 $7$ & \multirow{2}{*}{$x^2 -xz +2y^2 -2yz +5z^2$} & \multirow{2}{*}{10,13,14,30,40 \dots}& 30 \tabularnewline
 $\neq 5,7$ & & & 14 \tabularnewline
\hline
\end{tabular}
\captionsetup{size=small}
\caption{Ternary forms used for the switch.}\label{table:ternaryfortheswitch}
\end{table}

\section{Methods}

The next step is verifying whether the basic and the auxiliary escalators
constructed in the previous section represent every positive integer coprime to $p$.
We do this by a combination of four methods,
as was done in \cite{BH,DeBenedetto,Rouse},
but we also include an initial test with the 290-theorem,
which allows us to easily discard a big portion of our forms.
We will refer to such a test as `Method zero'. In Section \ref{S4.1}
we discuss its effects.
Methods 1 and 2, algebraic in nature, are described in Sections
\ref{S4.2} and \ref{S4.3}, respectively.
Methods 3 and 4, of a more analytic flavour,
will be discussed in later sections, after we introduce
certain required notions about modular forms.
The implementation of Methods 1--4 is based on the \textsc{magma}
code from~\cite{DeBenedetto,rouse:code}, appropriately modified,
see~\cite{publiccode}.

\subsection{Method zero: initial check}\label{S4.1}

If a positive definite integral quadratic form represents all integers up to 290,
then we know it is universal by Bhargava and Hanke's theorem \cite[Theorem 1]{BH}.
It turns out that this happens for many of our forms:
around six thousand forms for any given prime, see Table~\ref{table:methods}.
When $p\geq 11$, this amounts to more than 92\% of all our basic forms!
For completeness, we note that this also applies to 5442 out of 24888
quaternary escalators when $p=2$, and 6339 out of 9611 escalators when $p=3$.

We also know \cite[Theorem 3]{BH} that if a form represents everything below $290$
then it represents everything above that; and the same holds
with $15$ in place of $290$ if the form is classical \cite[\S6(iv)]{Bhargava}.
Among our quaternary escalators, there is none representing all of $1,\dots,289$,
but there are four classical forms representing $1,\dots,14$ when $p\geq 7$.
Since these forms become universal with one more escalation step,
they do not need to be analyzed further.

The above checks are extremely quick to run:
for each prime, testing all the basic quaternary
escalators takes less than 12 seconds in \textsc{pari/gp}.
To appreciate the time saving, note that Method zero works
on Form 3391 in \cite{DeBenedetto,rouse:code}, which required otherwise
22h of computation time using Method 3 \cite[p.441]{DeBenedetto}.
Similarly, it works on Form 22145
in \cite{Rouse}, which required otherwise to work with
a space of modular forms of dimension 936 with 19 Galois conjugacy classes
and required almost a day of computation \cite[p.1736]{Rouse}.

\subsection{Method 1: coprime-universal sublattices}\label{S4.2}

Given a quaternary lattice $L$, we check if there is a sublattice $L'\subseteq L$
whose quadratic form is one of those listed in Table \ref{table:ternaryforms}
or Table \ref{provablyternary}. Since all of these forms represent every
positive integer coprime to $p$ (we are assuming Conjecture \ref{intro:conjecture}),
then if such $L'$ exists we deduce that $L$ is coprime-universal.

\begin{example}
The quaternary form
\[
Q(x,y,z,w) = x^2 +xz +6xw + 2y^2 +2yz + 7z^2 +3zw + 14w^2
\]
appears as Form~6179 among the basic quaternary escalators when $p=5$
(here and below, numbering refers to~\cite{publiccode}).
A local inspection shows that $Q$ fails to represent all integers
of the form $5^ru$ with $r$ odd and $u\equiv\pm 2\pmod{5}$. We have
\[
Q(x-3z,y-z,0,z) = x^2 + 2y^2 +4yz + 7z^2,
\]
which is the fourth ternary form in Table \ref{provablyternary}
and is coprime-universal by Theorem~\ref{S3:thm}.
It follows that $Q$ is coprime-universal.
\end{example}

\begin{example}
The quaternary form
\[
Q(x,y,z,w) = x^2 +xz +3xw + 2y^2 +2yz +7yw + 7z^2 +5zw + 14w^2
\]
appears as Form~6292 among the basic quaternary escalators when $p=5$.
A local inspection shows that $Q$ fails to represent all integers
of the form $5^ru$ with $r$ odd and $u\equiv\pm 1\pmod{5}$. We have
\[
Q(x-z,y-z,0,z) = x^2 +xz + 2y^2 +3yz +7z^2,
\]
which is the first ternary form associated to $p=5$ in
Table~\ref{table:ternaryforms}.
Assuming Conjecture~\ref{intro:conjecture}
such a form is coprime-universal, hence so is $Q$.
\end{example}

Method 1 is very fast. Unfortunately, it does not apply to
any quaternary form when $p\geq 7$.
When $p=5$, it applies to 212 basic forms and one auxiliary form.

\subsection{Method 2: regular sublattices}\label{S4.3}

By the works of Jagy--Kaplansky \cite{913} and Oh~\cite{Oh},
it is known that there are 899 integral positive definite regular
ternary quadratic forms. Additionally, fourteen
more ternary forms were conjectured to be regular in~\cite{913}
and this was proved under the assumption of the Generalized Riemann Hypothesis
by Lemke-Oliver \cite{LO}. No other ternary form is regular.

Using the list of 899 provably regular ternary forms, we check if a quaternary lattice $L$
contains a sublattice $L'\subseteq L$ whose quadratic form, say $T(x,y,z)$, appears in such a list.
When this happens, we complete $T$ to a quaternary form by adding a vector
orthogonal to $T$ in $L$ and obtain a quadratic form
\begin{equation}\label{2606:eq002}
T(x,y,z) + dw^2,
\end{equation}
for some positive integer $d$. Whether an integer $m$ is represented
by $T$ depends only on congruence conditions.
We then show that \eqref{2606:eq002} represents all positive
integers except a finite list that can be detected by congruence conditions.
The way to obtain such a list explicitly is best explained by means of an example.

\begin{example}
Consider the form
\[
Q(x,y,z,w) = x^2 + 2y^2 + 5z^2 + 5zw + 10w^2,
\]
which appears as Form~6611 among the basic escalators when $p=7$.
Clearly, $Q$ represents $T(x,y,z)=x^2+2y^2+5z^2$, which is regular.
Its orthogonal complement in $Q$ is generated by a vector of norm $35$.
Hence, the form
\begin{equation}\label{2606:eq001}
T(x,y,z) + 35w^2
\end{equation}
corresponds to a sublattice of $Q$
(indeed, $Q(x,y,z+w,-2w)$ returns \eqref{2606:eq001}).
Local considerations tell us that $T$ represents every positive integer
except those of the form $5^ru$ with $r$ odd and $u\equiv \pm 2\pmod{5}$.
Thus, the classes $10$ and $15\pmod{25}$ are not represented.
On the other hand, \eqref{2606:eq001} represents $35$ and $T$
represents everything $\equiv 5\pmod{25}$, so \eqref{2606:eq001}
represents all integers $\equiv 15\pmod{25}$ and strictly greater than $15$.
We argue similarly for the class $10\pmod{25}$; now,
$T$ does not represent everything $\equiv 0\pmod{25}$,
but it does represents $25k$ with $5\!\nmid\! k$.
Therefore, \eqref{2606:eq001} represents all integers $\equiv 10,60,85,110\pmod{125}$
and strictly greater than $10$. As for the class $35\pmod{125}$:
for every $k\geq 1$ we can write $125k+35=5(25k-21)+35\cdot 2^2$;
since integers of the form $5(25k-21)$ are represented by $T$ for every $k$,
we deduce that every positive integer $\equiv 35\pmod{125}$ is represented by \eqref{2606:eq001}.

Altogether, we showed that \eqref{2606:eq001} represents all positive integers
except $10$ and~$15$. We then go back to the original form $Q$
and find that it represents $10$ but does not represent $15$.
Therefore, $Q$ has the single exception $15$ and it represents
all other positive integers coprime to $7$.
\end{example}

Returning to the general case, we do a local inspection on $T$
and initialize a modulus $M$ such that, for every residue class $a\pmod{M}$,
either $T$ represents all squarefree integers in that class or no integer.
Then, in each arithmetic progression not represented by $T$, we find
the smallest integer represented by $T(x,y,z)+dw^2$ with $T(x,y,z)\neq 0$.
This gives a modulus $M'\geq M$ such that
every number in the residue class $a\pmod{M'}$ is represented.
If $M'=M$, then we are finished with this residue class.
If $M'>M$, then we repeat the argument to the classes
$a+kM\pmod{M'}$ that contain squarefree integers.
Once all the residue classes have been checked, we are left
with a finite list of exceptions for $T(x,y,z)+dw^2$.
We then go back to the original quaternary form $Q$
to see if $Q$ represents these numbers.
For more details, we refer to \cite[p.4]{BH},
\cite[pp.438--439]{DeBenedetto} and \cite[pp.1731--1732]{Rouse}.

Method 2 runs relatively fast: for each form it does not take more than a few minutes.
It works on a fair amount of forms, too. As an example, it works 
on 74 out of 148 forms when $p=11$, hence corresponding to 50\%
of those forms to which Method zero and Method 1 do not apply,
and on 1171 out of 1312 forms when $p=7$, leaving only 141 forms.
For a detailed account of how many forms are treated with Method 2,
see Table \ref{table:methods}.

\begin{table}[!ht]
\footnotesize
\renewcommand\arraystretch{1.15}
\begin{tabular}{r|*{6}{>{\color{\darkgreen}}c|>{\color{\darkmagenta}}c|}}
\hline
$p$ & \multicolumn{2}{c|}{Forms} & \multicolumn{2}{c|}{Method zero} & \multicolumn{2}{c|}{Method 1} & \multicolumn{2}{c|}{Method 2} & \multicolumn{2}{c|}{Method 3} & \multicolumn{2}{c|}{Method 4}\\
\hline
 5 & 10429 & 1188 & 6145
                        & 5 & 212 & 1 & 1824 & 87 & 1201 & 541 & 1047 & 554 \\
\hline
 7 &   7558 & 486 & 6246
                        & 0 &  0  & 0 & 1171 & 19 & 48 & 286 & 93 & 181 \\
\hline
11 &   6339 & 104 & 6191
                        & 0 &  0  & 0 & 74 & 10 & 39 & 57 & 35 & 37 \\
\hline
13 &   6844 & 104 & 6306
                        & 0 & 0 & 0 & 87 & 10 & 265 & 57 & 186 & 37 \\
\hline
17 &   6315 & 104 & 6167
                        & 0 & 0 & 0 & 74 & 10 & 39 & 57 & 35 & 37 \\
\hline
19 &   6567 & 104 & 6418
                        & 0 & 0 & 0 & 75 & 10 & 39 & 57 & 35 & 37 \\
\hline
23 &   6139 & 104 & 5990
                        & 0 & 0 & 0 & 75 & 10 & 39 & 57 & 35 & 37 \\
\hline
29 &   5893 & 104 & 5708
                        & 0 & 0 & 0 & 74 & 10 & 72 & 57 & 39 & 37 \\
\hline
31 &   5745 & 104 & 5597
                        & 0 & 0 & 0 & 75 & 10 & 38 & 57 & 35 & 37 \\
\hline
37 &   6567 & 104 & 6418
                        & 0 & 0 & 0 & 75 & 10 & 39 & 57 & 35 & 37 \\
\hline
\end{tabular}
\captionsetup{size=small,width=110mm}
\caption{How many forms are treated with each method.}\label{table:methods}
\end{table}

\section{Analytic methods}\label{S5}

The last two methods rely on the theory of modular forms.
In this section, we first recall some properties of theta series
and their decomposition into Eisenstein and cuspidal part.
Then, we explain how to effectively use such a decomposition
in Methods~3 and 4.

Let $Q$ be a positive definite quaternary integral quadratic form.
As in Section~\ref{S2}, we write $Q(x)=\frac{1}{2}x^tAx$,
where $A$ has integer coefficients.
The number $\det(A)$ is called the \emph{discriminant}~of~$Q$, whereas
the \emph{level} of $Q$ is the smallest positive integer~$N$ such that
the matrix $NA^{-1}$ has integer entries and even diagonal.
Clearly, the discriminant has this property, but a smaller integer often works.

Denoting $r_Q(n)=|\{x\in\ZZ^4:Q(x)=n\}|$, we build the theta series ($q=e^{2\pi iz}$)
\[
\theta_Q(z) = \sum_{n=0}^{\infty} r_Q(n)q^n,
\]
which is a modular form on $\Gamma_0(N)$ of weight 2 and
character $\chi_D(n)=\big(\frac{D}{n}\big)$, see e.g.~\cite[Theorem 10.8]{topics}.
We can therefore decompose $\theta_Q$ as the sum of its Eisenstein
and cuspidal projections:
\[
\theta_Q(z) = E(z) + C(z) = \sum_{n\geq 0} a_E(n)q^n + \sum_{n\geq 1} a_C(n)q^n.
\]
With the goal of showing that the Eisenstein coefficient is larger than
the absolute value of the cuspidal coefficient (so that $r_Q(n)>0$),
let us calculate explicit lower bounds on the former and upper bounds on the latter.

The coefficient $a_E(n)$ is a product of local densities:
\[
a_E(n) = \prod_{p\leq \infty} \beta_p(n).
\]
If $n$ is locally represented, then $\beta_p(n)\neq 0$ for all $p$.
Moreover, if $n$ is squarefree, then~\cite[Theorem 5.7(b)]{Hanke} (see also~\cite{Yang})
\[
\beta_p(n)=
\begin{cases}
\frac{\pi^2n}{\sqrt{\det(A)}} & \text{if $p=\infty$}\\
1-\frac{\chi_D(p)}{p^2} & \text{if $p\nmid N$ and $p\nmid n$}\\
\frac{(p-1)(p^2+(1+\chi_D(p))p+1)}{p^3} & \text{if $p\nmid N$ and $p|n$.}
\end{cases}
\]
As a consequence, we have the lower bound (cf.~\cite[p.~440]{DeBenedetto})
\[
a_E(n) \geq \frac{\pi^2n}{\sqrt{\det(A)}L(2,\chi_D)}\left(\prod_{p|N}\frac{\beta_p(n)}{1-\chi_D(p)p^{-2}}\right)\prod_{\substack{p\,\nmid N\\p|n,\chi_D(p)=-1}}\frac{p-1}{p+1}.
\]
Next, we compute the densities $\beta_p(n)$ for all $p|2N$ and all classes of $\ZZ_p/\ZZ_p^2$ that contain
squarefree integers, obtaining an effective constant $C_E$, depending on~$Q$, such that
\begin{equation}\label{0807:eq001}
a_E(n) \geq C_E n \prod_{\substack{p|n\\p\,\nmid N,\chi_D(p)=-1}}\frac{p-1}{p+1}
\end{equation}
for all squarefree positive integers $n$ locally represented by $Q$.
This accounts for the Eisenstein contribution.

As for the cuspidal part, we can explicitly decompose
\begin{equation}\label{0807:eq003}
C(z) = \sum_{d|N} \sum_{i=1}^{s} \sum_{e|\frac{d}{\mathrm{cond}\,\chi_D}} c_{d,i,e} \, g_i(ez),
\end{equation}
where $c_{d,i,e}$ are complex numbers and $\{g_i\}_{i=1}^s$ is a basis of
normalized Hecke eigenforms for the new subspace of $S_2(\Gamma_0(d),\chi_D)$.
By Deligne's bound, the $n$th Fourier coefficient of each $g_i(z)$
is bounded in absolute value by $d(n)\sqrt{n}$. Thus, if we set
\begin{equation}\label{0807:eq002}
C_Q = \sum_{d|N} \sum_{i=1}^{s} \sum_{e|\frac{d}{\mathrm{cond}\,\chi_D}} \frac{|c_{d,i,e}|}{\sqrt{e}},
\end{equation}
we have $|a_C(n)|\leq C_Q d(n) \sqrt{n}$.

Combining \eqref{0807:eq001} and \eqref{0807:eq002}, we deduce that
there is a constant $F$ (namely, $\frac{C_Q}{C_E}$ in the above notation)
such that if
\[
F_4(n) := \frac{\sqrt{n}}{d(n)} \prod_{\substack{p|n\\p\,\nmid N,\chi_D(p)=-1}}\frac{p-1}{p+1} > F,
\]
then $n$ is represented by $Q$.
In order to determine all the squarefree integers represented by~$Q$,
it suffices to list all squarefree $n$ for which $F_4(n)<F$
and check if they are represented.

\subsection{Method 3}\label{S5.1}

The decomposition \eqref{0807:eq003} can be rather time-consuming
to calculate, as well as heavy on memory, because it requires solving
a linear system on a vector space of large dimension.
As an example, the form $Q(x,y,z,w)=x^2+2y^2+7z^2+xw+yw+21w^2$,
which appears as Form~1130 among the auxiliary escalators when $p=5$,
has discriminant $D$ and level $N$ equal to $4620$
and the dimension of $S_2(\Gamma_0(N),\chi_D)$ is 1136.
This is the largest dimension among all the forms that
are not treated with one of the previous methods.
In the third method we use a theoretical shortcut to obtain a reasonable
upper bound on $C_Q$ without having to go through linear algebra.
We sketch here the main ideas and refer to \cite[pp.~1732--1734]{Rouse}
and \cite[\S 6]{DeBenedetto} for more details.

We require that $D$ is a fundamental discriminant, hence $\chi_D$ is primitive
and the decomposition \eqref{0807:eq003} reduces to the single sum over $i=1,\dots,s$,
because the new subspace of $S_2(\Gamma_0(N),\chi_D)$ is the entire space.
Note that, if there exist two positive numbers $A,B$ such that
\[
\langle C,C\rangle \leq A \quad\text{and}\quad \langle g_i,g_i\rangle \geq B,\quad \forall\;i=1,\dots,s
\]
(here $\langle f,g\rangle$ denotes the Petersson inner product),
then by Cauchy--Schwarz and the orthogonality of the $g_i$'s we obtain
\begin{equation}\label{2509:eq001}
C_Q
=
\sum_{i=1}^{s} |c_i|
\leq
\sqrt{s} \left(\sum_{i=1}^{s} |c_i|^2\right)^{1/2}
\leq
\sqrt{\frac{s}{B}}\sqrt{\langle C,C\rangle} \leq \sqrt{\frac{As}{B}}.
\end{equation}

Rouse proved \cite[Proposition 11]{Rouse} that if $g_i$
has no complex multiplication (CM in short), then
\[
\langle g_i,g_i \rangle \geq \frac{3}{208\pi^4\log N}\prod_{p|N}\left(1+\frac{1}{p}\right)^{-1}.
\]
For any given $Q$, one can enumerate all CM-forms in $S_2(\Gamma_0(N),\chi_D)$
and verify that the same bound holds for them.

An upper bound for $\langle C,C\rangle$ can be obtained by passing to $C^*$,
defined as follows: we first take the quadratic form with Gram matrix $NA^{-1}$;
then $C^*$ is the cuspidal projection of its theta series.
By \cite[Proposition 15]{Rouse}, we have $\langle C,C\rangle=N\langle C^*,C^*\rangle$.
Furthermore, if we write $C^*(z)=\sum_{n\geq 1} a(n)q^n$, then
by \cite[Propositions 14,15]{Rouse}, we have
\[
\langle C^*,C^*\rangle = \frac{1}{[\mathrm{SL}_2(\ZZ):\Gamma_0(N)]}
\sum_{n=1}^{\infty} \frac{2^{\omega(\gcd(n,N))}a(n)^2}{n}\sum_{d=1}^{\infty} \psi\left(d\sqrt{\frac{n}{N}}\right),
\]
where $\psi(x)=-\frac{6}{\pi}K_1(4\pi x)+24x^2K_0(4\pi x)$ and $K_0,K_1$ are
the Bessel functions of the second kind and order zero and one, respectively.

From here, arguing as in \cite[p.1733]{Rouse} and \cite[p.442-443]{DeBenedetto}
and splitting the sum over the first $15N$ terms and the tail,
we obtain an inequality of the form
\[
\langle C^*,C^*\rangle
\leq
C_1 + 1.71\cdot 10^{-18}N^{7/4} \left(1+\frac{1}{15N}\right)^{3/2} \frac{s}{B} \langle C^*,C^*\rangle
=
C_1 + \lambda\langle C^*,C^*\rangle,
\]
say. If $\lambda<1$, then one can solve the inequality and obtain an upper bound on
$\langle C^*,C^*\rangle$, from which we recover an upper bound on $\langle C,C\rangle$
as desired.

\begin{example}
The form $Q(x,y,z,w)=x^2+2xz+7xw+2y^2+yz+7z^2+7zw+14w^2$
appears as Form~5251 among the basic quaternary escalators when $p=5$.
Its determinant is $D=329$, a fundamental discriminant, and its level is $329$, too.
The associated cuspidal subspace $S_2(\Gamma_0(329),\chi_{329})$
has dimension $30$ and contains 10 CM-forms,
whereas the Eisenstein subspace has dimension 4
(label 329.2.c on the \textsc{lmfdb} online database \cite{lmfdb}).

In the notation of \eqref{2509:eq001}, we have $s=30$.
Moreover, $B=4.031123835\times 10^{-5}$ is a lower bound on $\langle g_i,g_i\rangle$
for all $i$, while $\langle C,C\rangle\leq A=0.02948733638$.
It follows that $C_Q\leq 148.1376086$.
On the other hand, the Eisenstein constant is $C_E=69/47\geq 1.4680851$,
hence we conclude that $Q$ represents every positive integer $n$ that is squarefree, coprime to $5$
and such that $F_4(n)>101.9143809$. This calculation requires 13 seconds.

There are 2408675 squarefree integers coprime to $5$ having $F_4(n)\leq 101.9143809$
and all of them are represented by $Q$. Therefore, $Q$ represents
all squarefree positive integers coprime to $5$ (note, however, that $Q$ does not represent 5).
This calculation takes 17 seconds.
\end{example}

\subsection{Method 4}\label{S5.2}

Finally, when $D$ is not a fundamental discriminant,
Method 3 cannot be applied and so we return
to the determination of $C_Q$ by means of \eqref{0807:eq003}.

\begin{example}
The form $Q(x,y,z,w)=x^2+xz+2xw+2y^2+5z^2+10w^2$
appears as Form~5550 among the basic quaternary escalators when $p=7$.
It has determinant 136 and level 680. The associated cuspidal subspace
$S_2(\Gamma_0(680),\chi_{85})$ has dimension 100.
Bounding as in \eqref{0807:eq002}, we estimate the cuspidal constant
by $C_Q\leq 15.8324995$.
The Eisenstein subspace has dimension $16$ and we compute the Eisenstein
constant $C_E=4/45$, which leads to $F=179.8967754$.
This calculation takes 22 seconds.

There are 3384052 squarefree integers coprime to $7$ having $F_4(n)\leq F$
and all of them are represented by $Q$. Therefore, $Q$ represents
all squarefree positive integers coprime to $7$ (note, however, that $Q$ does not represent 14 and 56).
This calculation takes 27 seconds.
\end{example}

Among all the quaternary escalators (basic and auxiliary) that are treated with
the analytic methods, Method 3 works for a little over 50\% of the forms
and Method 4 works on the remaining ones, see Table~\ref{table:methods}.
Overall, the analytic methods require several days of computation to complete.

\section{Proof of Theorem \ref{theo:MAIN}}\label{S6}

Fix a prime $5\leq p\leq 37$. Let $Q$ be a positive definite integral quadratic form
and assume that $Q$ represents all integers from the set $S_p$ from Table~\ref{table:Sp}.
By the escalation process, we have the following possibilities:
\begin{itemize}
   \item[(i)\;]  $Q$ represents one of the ternary forms
                 in Table \ref{table:ternaryforms}
                 or Table \ref{provablyternary};
   \item[(ii)\,] $Q$ represents one of the basic quaternary escalators,
                 see Section~\ref{S3.1} and Table~\ref{table:quat};
   \item[(iii)]  $Q$ represents a quaternary escalator with local obstructions;
                 there are 22 such forms when $p=5$ and five forms when
                 $7\leq p\leq 37$, see Table~\ref{table:quat}.
\end{itemize}

In the first case, by Conjecture \ref{intro:conjecture} and
Theorem~\ref{S3:thm}, $Q$ is coprime-universal and we are done.

In case (ii), we look at the possible squarefree integers
not represented by $Q$. Quaternary forms can have several
squarefree exceptions coprime to $p$. If they are all in $S_p$,
then by a finite number of escalations we arrive at a
coprime-universal sublattice of $Q$ and we are done.
If a quaternary form $q$ has exceptions $s_1,\dots,s_N\notin S_p$,
we consider all rank-five escalations of $q$ by its truant $t\in S_p$
and verify that they represent all of $s_1,\dots,s_N$.
Since $Q$ represents $t$, it contains one of these
five-dimensional lattices, hence again, up to possibly escalate
by elements in $S_p$, we deduce that $Q$ is coprime-universal.

\begin{example}
When $p=5$, the form $x^2+xz+2xw+2y^2+7z^2+14w^2$,
which appears as Form~7555, misses $5,10,20,26,122$ and represents
all other positive integers. All five-dimensional escalations by $26$
represent $122$ and so they are coprime-universal.
\end{example}

\begin{example}
Similarly, Form~9704, namely $x^2+xy+xz+2y^2+yz+3z^2+19w^2$,
misses $10,38,57$ and represents all other positive integers.
All escalations by $38$ represent $57$ and so they are coprime-universal.
\end{example}

No other basic form requires this step when $p=5$. Furthermore,
when $p\geq 7$ no basic form requires this step. We note, however,
that $203$ and $290$ appear as truants of five-dimensional escalators.
Indeed, we always have the form $x^2+xz+2y^2+3yz+5z^2+29w^2$,
which misses $145,203,290$ and represents all other positive integers.
This form is equivalent to $L_{145}$ in \cite[(4)]{BH} and there exist
rank-five escalators of it with truant $203$ and $290$,
as showed in \cite[(3) and p.14]{BH}.
This concludes the proof for case (ii).

In case (iii) we must examine the auxiliary lattices,
looking again for possible squarefree integers not represented by $Q$.
When $p=5$, out of 1188 auxiliary forms, we find nine to take care of;
apart from the elements in $S_p$, their squarefree exceptions are:
 91 (Forms 1142,1182),
101 (Forms 823,825),
122 (Form 682),
154 (Forms 762,810),
158 (Form 1114) and
394 (Form 388).
Escalating all of these forms by their common truant $14$,
we find that every five-dimensional escalator so contructed
represents these exceptions and is therefore coprime-universal.

When $p=7$, out of 486 auxiliary forms, we find zero to take care of
(all of their exceptions are either divisible by 7, or non-squarefree,
or already in $S_7$). The same is true for $p=11,\dots,37$.
This concludes the proof of Theorem \ref{theo:MAIN}.

\section{Proof of Theorem \ref{theo:MAIN1}}\label{sec:classical}

We apply again the escalation process, selecting only integer-matrix escalators.
We start by escalating the the one-dimensional lattice with Gram matrix $(2)$
and associated quadratic form $x^2$. There are two classical binary escalators,
corresponding to $x^2+y^2$ and $x^2+2y^2$ (cf.~\eqref{eq:22}).

\subsubsection*{$p=5$}
Escalating by 3 and 7 respectively the binary forms above,
we obtain 12 classical ternary escalators. Two of them
are coprime-universal by Theorem~\ref{S3:thm}
(they are the second and third form in Table~\ref{provablyternary}).
Escalating the other $10$ produces $443$ classical quaternary lattices,
all of which represent locally everything coprime to $5$.
Applying Method zero we find that $179$ forms are universal.
The remaining 264 forms are treated with Method 1 (8 forms), Method 2 (64 forms)
or Methods 3 and 4 (64 forms).
Overall, we find that $414$ out of $443$ forms are coprime-universal;
$26$ forms have exactly one exception and will become universal
at the next escalation step.
The remaining three forms have exactly two exceptions each.
Since all the escalations by their truant represent their second exception,
we arrive again at a coprime-universal lattice in one step.

\subsubsection*{$p=7$}
We escalate by 3 and 5 the binary forms above,
obtaining 9 classical ternary lattices. None of them is coprime-universal,
so we escalate again producing $381$ classical quaternary lattices.
Of these, $368$ are coprime-universal, while the remaining $13$ only miss
one integer coprime to $7$ and so they will become coprime-universal
at the next escalation step.

\section{Proof of Corollary \ref{intro:corollary}}\label{S8}

When $t\in S_1$, there exists a quadratic form representing all positive integers except $t$
by \cite[Theorem 2]{BH}. Similarly, when $t'\in S_1'$, there is a classical form
representing all positive integers except $t'$ by \cite[\S6(iii)]{Bhargava}.
Therefore, by Theorem~\ref{theo:MAIN} and Theorem~\ref{theo:MAIN1}
there is nothing to prove when $p\geq 11$.
Also, when $p=2,3$, the result was proved in \cite[Corollary 2]{DeBenedetto} and \cite[Corollary 3]{Rouse}.
Therefore, it suffices to give a proof when $t\in S_5\cup S_7$ with $t\notin S_1$
and when $t'\in S_5'\cup S_7'$ with $t'\notin S_1'$.

Now, for $p\in\{5,7\}$ and each value of $t\in S_p-S_1$,
there exists a quadratic form $Q$ representing
all positive integers coprime to $p$ and strictly less than $t$
(in other words, $Q$ has truant $t$), see Table~\ref{table:forms4truants}.
Then, the form
\[
Q(x) + (t+1)(a^2 + b^2 + c^2 + d^2) + \sum_{i=1}^{t-1} (t+1+i) x_i^2
\]
can be easily seen (by Lagrange’s four square theorem) to represent
every positive integer larger than $t$.
This proves the corollary for non-classical forms.
As for the classical case, the argument is the same, but one selects
$Q$ to be classical in Table~\ref{table:forms4truants}.

\section{Proof of Corollary \ref{intro:cor2}}\label{S81}

\subsection{Classical Forms}

Let $Q$ be a classical quadratic form representing
every positive integer coprime to $5$ less than $t'_{\max}=58$.
Then $Q$ represents a classical escalator lattice among the ones
we found (when working on $p=5$) in Section \ref{sec:coprime}.
Of those, the only one missing $58$ is the basic quaternary form
number $1313$, namely
\[
x^2+2y^2+7z^2+26w^2+4xz.
\]
This form was analyzed with Method 2 and fails to represent
only the two integers $10$ and~$58$.
The desired result now follows by remembering that
we are focusing on integers coprime to~$5$.
Similarly, let $Q$ be a classical quadratic form
representing every positive integer coprime to $7$ less than $t'_{\max}=78$.
Then $Q$ represents a classical escalator lattice from Section~\ref{sec:coprime}.
Only two of those miss $78$: the basic quaternary forms number $226$, namely
\[
x^2+2y^2+5z^2+30w^2+4xz+4xw+4zw,
\]
and number $3267$, namely
\[
x^2+2y^2+5z^2+30w^2+2xz+4xw+4zw.
\]
Both of these forms were analyzed with Method 2 and fail to represent
only the integers $14$ and $78$.
The desired result now follows by remembering that
we are focusing on integers coprime to $7$.
When $p\geq 11$, the result follows from \cite[(iv), Section 6]{Bhargava}.
When $p=2,3$, see \cite{Rouse,DeBenedetto}.

\subsection{Non-classical forms}

Let $5\leq p\leq 37$, $p\neq 29$. We have $t_{\max}=203$ when $p=5$
and $t_{\max}=290$ when $p\neq 5$.
Let $Q$ be a quadratic form representing every positive integer
coprime to $p$ less than $t_{\max}$.
Then $Q$ represents an escalator lattice from Section~\ref{sec:coprime}.
If $Q$ represents a provably coprime-universal ternary escalator
(that is, a ternary form from Table~\ref{provablyternary}), then we are done.
Similarly, if $Q$ represents a ternary escalator from Table~\ref{table:ternaryforms},
then we are done conditionally on Conjecture~\ref{intro:conjecture}.
Assuming that $Q$ represents a quaternary escalator,
we find that in all cases the only quaternary escalator not representing $t_{\max}$ is
\[
x^2+2y^2+7z^2+29w^2+3xz+3yz,
\]
which is equivalent to $L_{145}$ in \cite[(4)]{BH} and
misses the three integers $145,203$ and $290$.
After possibly escalating again, we find a sublattice of $Q$ missing
$t_{\max}$ and representing every other positive integer coprime to $p$.

Incidentally, we note that when $p=3$ the result for non-classical forms appears to be missing
in \cite{DeBenedetto}. Yet, the same behaviour as above occurs: if $Q$ represents
all positive integers coprime to $3$ less than $t_{\max}=290$, then
either it represents a coprime-universal ternary escalator from \cite[Theorem 4]{DeBenedetto},
or it represents one of $8894$ quaternary escalators, among which exactly
one, equivalent to $L_{145}$, misses $290$.

When $p=29$, we have $t_{\max}=110$. Let $Q$ be a quadratic form
representing every positive integer coprime to $29$ less than $t_{\max}$.
Then $Q$ represents an escalator lattice from Section~\ref{sec:coprime}.
If $Q$ represents the form $Q_{29}$ appearing in Table~\ref{table:ternaryforms}
then we are done conditionally on Conjecture~\ref{intro:conjecture}.
Assuming that $Q$ represents a quaternary escalator, we find that
the only one missing $110$ is the basic quaternary form number $5626$, namely
\[
x^2+y^2+3z^2+xz+22w^2.
\]
This form was analyzed with Method 2 and it fails to represent
only the integer $110$, from which the corollary follows.

\section{Proof of Theorem \ref{theo:partial1}}\label{sec:elementary}

In this section we prove Theorem \ref{theo:partial1} by applying
an elementary technique previously used by Kaplansky
\cite[Section 4]{Kaplansky} and a method of embedded genera,
which proved useful in Bhargava's proof of the 15-theorem
\cite[p.~30]{Bhargava}.

The first technique requires auxiliary regular ternary forms.
Given a quadratic form $Q$ and an integer $N$ which we want to prove is represented by $Q$,
we show that there exist integers $k,d\geq 1$ and a $3\times 3$ integer matrix $M$ such that
\[
kQ\left(\frac{1}{d}Mv\right)=L(v),
\]
where $v=(x,y,z)$ and $L$ is a regular form representing $kN$.
Then, among all vectors for which $L(v)=kN$, we show that we can choose one such that
$Mv\equiv 0\pmod{d}$, obtaining that $kQ$ represents $kN$, i.e.~$Q$ represents $N$.
In fact, it suffices that $gMh v\equiv 0\pmod{d}$ for some automorph $g$ of $Q$ and $h$ of $L$.

This approach works for the forms $Q_{\twoa},Q_{\twob},Q_{\twoc},Q_{\fivea}$ and $Q_{\fiveb}$.
These forms are in a genus with more than one form and are not regular.
For instance, the genus of $Q_{\fivea}$ has size $2$ and a representative of the other class is
$R=2x^2+3y^2+3z^2-xy+2xz+2yz$.
Altogether, the genus $\{Q_{\fivea},R\}$ represents every integer not of the form $5^{1+2s}m$ with
$m \equiv \pm 1 \pmod 5$ and so, in particular, every integer coprime to~$5$.
Conjecture~\ref{intro:conjecture} predicts that $Q_{\fivea}$ itself represents such integers.
Note that $Q_{\fivea}$ does not represent $10$ and $R$ does not represent $1$
(but it seems to represent every other positive integer coprime to $5$).
Similar considerations apply to the other forms in Table~\ref{table:ternaryforms}.

In the second method we show that a form $Q$ contains, as proper
sublattices, the representatives of a whole genus; as a consequence,
we deduce that $Q$ represents itself all the integers represented
by such genus.
We use this approach on $Q_{\fiveb},Q_{23},Q_{29}$ and $Q_{31}$.

\subsection{Proof of Theorem \ref{theo:partial1}}
Let $N$ be a squarefree positive integer with $5\nmid N$ and $N \equiv 2 \pmod 3$.
We have
\begin{equation}\label{eq:form1}
Q_{\fivea}\left( x+\frac{y+z}{3},-z,\frac{y+z}{3}\right)=x^2+y^2+2z^2+xy+xz+yz.
\end{equation}
Let $L(x, y, z)$ denote the form on the right; $L$ is regular
(it is in genus $1$) and represents every integer coprime to $5$.
Therefore, there are $x, y, z \in \mathbb{Z}$ such 
\begin{equation}\label{eq:cond1}
L(x,y,z)=N.
\end{equation}
If we can show that $x,y,z$ can be picked so to have $y + z \equiv 0 \pmod 3$,
then \eqref{eq:form1} gives a representation of $N$ by $Q_{\fivea}$.
Observe that, if $(x, y, z)$ satisfies \eqref{eq:cond1},
then we also have the representations $(y,x,z)$ and $(x,-x-y-z,z)$.
Hence, it suffices to show that there is a representation
satisfying at least one of the congruences
\begin{equation}\label{eq:cond2}
x + y \equiv 0 \pmod 3,\quad
x + z \equiv 0 \pmod 3,\quad
y + z \equiv 0 \pmod 3.
\end{equation}
We have
\[
x^2 + y^2 + (x + y + z)^2 \equiv 2L(x, y, z) \equiv 2N \equiv 1 \pmod 3,
\]
which forces
\[
(x^2, y^2, (x + y + z)^2) \in \{(1, 0, 0), (0, 1, 0), (0, 0, 1)\} \pmod 3
\]
and in all cases at least one of the congruences in \eqref{eq:cond2} is satisfied.

Assume now $N \equiv 0 \pmod 3$. We have
\begin{equation}\label{1803:eq001}
2Q_{\fivea}\left(\frac{x+y-4z}{3},\frac{x+z}{2},\frac{-x+2y+z}{6} \right) = x^2+2y^2+5z^2.
\end{equation}
The form on the right is regular and represents all integers coprime to $5$,
therefore it represents $2N$.
Moreover, given any representation $(x,y,z)$ we have
\begin{equation}\label{1803:eq002}
x+z\equiv x^2 + 2y^2 + 5z^2 \equiv 2N \equiv 0 \pmod 2.
\end{equation}
We also have
\[
x^2 + 2y^2 + 5z^2 \equiv 2N \equiv 0 \pmod 3,
\]
which forces one variable to vanish and the other two to be non-zero modulo~$3$
(for if all are zero then $9|N$).
Up to changing sign to one of the non-zero variables, we
have $x-2y-z \equiv 0 \pmod 3 $.
Combining this with \eqref{1803:eq001} and \eqref{1803:eq002},
we deduce that $Q_{\fivea}$ represents $N$, as claimed.

Moving to the form $Q_{\fiveb}$, assume $N$ is a squarefree positive integer coprime to $5$
and divisible by $11$, say $N=11N_0$ with $(N_0,11)=1$. We have
\begin{equation}\label{1803:eq003}
11Q_{\fiveb}\left(\frac{3x+y+3z}{11},\frac{x+y-2z}{11},\frac{y+z}{11}\right)
=
x^2+y^2+2z^2+xy+xz+yz.
\end{equation}
The form on the right is regular and represents every integer coprime to $5$.
In particular, it represents $N_0$ by means of some integer vector $(x_0,y_0,z_0)$.
On taking $(x,y,z)=(11x_0,11y_0,11z_0)$ in \eqref{1803:eq003},
we deduce that $Q_{\fiveb}$ represents $N$.

When $p=2$, we use the auxiliary form $L(x,y,z)=x^2+y^2+2z^2$, which represents
all positive odd integers. We have
\begin{equation}\label{1705:eq001}
19 Q_{\twoa}\left(\frac{x+y-6z}{19},\frac{3x+z}{19},-\frac{2y}{19}\right) = L(x,y,z).
\end{equation}
If $N=19N_0$ with $N_0$ odd, then $L(x_0,y_0,z_0)=N_0$ for some $(x_0,y_0,z_0)\in\ZZ^3$.
By \eqref{1705:eq001} it follows that $Q_{\twoa}(19x_0,19y_0,19z_0)=19N_0=N$.
For $Q_{\twob}$ and $Q_{\twoc}$ the argument is the same but we use the identities
\[
\begin{split}
31Q_{\twob}\left(\frac{4x+6z}{31},\frac{-3x-y+2z}{31},\frac{x-2y-z}{31}\right) &=L(x,y,z)\\
37Q_{\twoc}\left(\frac{5x+3y-4z}{37},\frac{-3x+2y-z}{37},\frac{y+3z}{37}\right) &=L(x,y,z).
\end{split}
\]

As for the method of embedded genera, we have
\[
\begin{split}
Q_{\fivea}(-x+y+z,-x-2z,-y+z) &= 3x^2+7y^2+11z^2+2xy+2xz-6yz,\\
Q_{\fivea}(-x+y-z,-x-y-2z,2z) &= 3x^2+3y^2+23z^2+2xy+2xz+2yz.
\end{split}
\]
The two forms on the right form a genus and altogether they represent
every squarefree positive integer $N$ coprime to $5$ with $N\equiv 3\pmod{4}$.
Similarly, we have
\[
\begin{split}
Q_{\fiveb}(2x+y+z,-2z,y) &= 4x^2 + 8y^2 + 9z^2 + 4xy + 4xz,\\
Q_{\fiveb}(x,2y,2z) &= x^2 + 8y^2 + 28z^2 +4yz.
\end{split}
\]
The two forms on the right form a genus, which represents
every squarefree positive integer $N$ coprime to $5$ with $N\equiv 1\pmod{4}$.

When $p=23$ we find that
$Q_{23}(2x+2y+z,x+2z,-x) = 4x^2 + 4y^2 + 9z^2 + 4xy + 4xz + 4yz = R$, say.
The genus of $R$ has size three, it is fully contained in $Q_{23}$
and it represents every squarefree integer $n$
coprime to $23$ and with $n\equiv 1\pmod{4}$.

When $p=29$ we find that
$Q_{29}(-x+2y+2z,-x-2z,x+y) = 4x^2 + 11y^2 + 12z^2 +4xy + 4yz=R$, say.
Again $R$ is in a genus of size three which is fully contained in
$Q_{29}$ and represents all squarefree integers coprime to $29$ in the class $3\pmod{8}$.

Finally, when $p=31$ we have
$Q_{31}(2x+y+2z,y-z,-y-z) = 4x^2 + 5y^2 + 8z^2 +4xz +4yz=R$,
which is in a genus of size three fully contained in $Q_{31}$
and representing all squarefree integers coprime to $31$
in the class $1\pmod{4}$.\qed

\begin{remark}\label{sec:elementary:rmk}
We did not find other auxiliary forms useful to apply Kaplansky's method.
On the other hand, the method of embedded genera provides more arithmetic progressions
than the ones stated in Theorem \ref{theo:partial1}, involving primes not dividing
the discriminant of the forms considered. For instance, $Q_{23}$ represents
a genus of size 6 (associated to the form $4x^2+7y^2+9z^2+2xy+6yz$)
which represents squarefree integers $n$ coprime to $23$
with $n\equiv 0,1\pmod{3}$ and a genus of size 7
(corresponding to $10x^2+12y^2+13z^2+2xy+10xz+8yz$)
giving the congruence
$n\equiv 0,2,3,5\pmod{7}$. Similarly, $Q_{29}$ represents the genus
of $9x^2+16y^2+19z^2+6xy+6xz+8yz$, of size $13$, giving the congruence $n\equiv 1\pmod{3}$ and
the genus of $19x^2+19y^2+19z^2+8xy+2xz+2yz$, of size 15, which gives $n\not\equiv 2,8\pmod{15}$.
\end{remark}

\section{Proof of Theorem \ref{thm:GRH}}\label{sec:GRH}

Denote by $Q_{\fivea},Q_{\fiveb}$ and $Q_{23},Q_{29},Q_{31}$ the forms appearing
in Conjecture \ref{intro:conjecture}, see Table \ref{table:ternaryforms}.
We prove here that they are coprime-universal under GRH.

\subsection{Strategy of proof}\label{sec:GRH.1}
Let $Q$ be any of our five ternary forms.
The associated theta series $\theta_Q$ decomposes into
Eisenstein series and cuspidal part, say
\begin{equation}\label{0303:eq002}
\theta_Q(z) = E(z) + C(z) = \sum_{n\geq 0} a_E(n)q^n + \sum_{n\geq 1} a_C(n)q^n.
\end{equation}
In essence, we are going to show that $a_E(n)$ dominates on $|a_C(n)|$,
in the same spirit as we did in Section \ref{S5}.
To do so, we exploit the fact that $a_E(n)$ and $a_C^2(n)$
are related to special values of $L$-functions.
Assuming GRH, these can be bounded explicitly
and produce the desired result as soon as $n$
is larger than $10^{9}$ roughly. For $n\leq 10^{9}$
we check with a computer that the result holds.

The difficulty with ternary forms is that $\theta_Q$
is a modular form of half-integral weight, where
much less is known compared to integral weight:
the available unconditional estimates on Fourier coefficients of
weight $\frac{3}{2}$ modular forms would only give
the desired conclusion when $n$ is greater than
a threshold of magnitude possibly as large as $10^{85}$ (see \cite[p.419]{OS}),
which cannot be reached with current computers.
To circumvent this issue, we relate $a_C^2(n)$ to the
$L$-function of an integral weight form by means of
Waldspurger's theorem, and then estimate the new object in weight 2,
where we can use explicit bounds, conditional on GRH, due to Chandee \cite{chandee}.

The Eisenstein coefficients $a_E(n)$ in \eqref{0303:eq002}
are a little easier to deal with, as they are understood
by means of Dirichlet $L$-functions.
Denoting by $D$ the discriminant of $Q$, for all squarefree positive integers $n$ we have
\begin{equation}\label{0303:eq003}
a_E(n) = \frac{24}{\pi\sqrt{D}}\,\sqrt{n}\,L(1,\chi_{-D'n})\prod_{p|2D} \delta_p(n,Q)(1-p^{-2})^{-1},
\end{equation}
where $D'=4D$ for all of our forms. Note that the Kronecker character $\chi_{-D'n}$ needs not be primitive at this stage.
By Dirichlet's class number formula, \eqref{0303:eq003} can be expressed in terms of class numbers
of imaginary quadratic fields, but we do not need that here.

\subsection{Waldspurger and Shimura}
Our key tools to understand the cuspidal coefficients $a_C(n)$ in \eqref{0303:eq002}
are the Shimura lift \cite{shimura} and Waldspurger's theorem \cite{waldspurger}.
The Shimura lift allows us to move from weight $3/2$ to weight $2$ modular forms.
Then, by Waldspurger's theorem, we can restrict
our attention to a finite set of squareclasses.

Let $S_{3/2}(\Gamma_0(4N),\chi)$ denote the vector space of cusp forms
of weight $3/2$ on $\Gamma_0(4N)$ with character $\chi$.

\begin{theorem}[Shimura lift]
Suppose that $f(z)=\sum_{n\geq 1} a(n)q^n \in S_{3/2}(\Gamma_0(4N),\chi)$.
For each squarefree positive integer $t$, let
\[
\big(\mathcal{S}_t(f)\big) (z) = \sum_{n=1}^{\infty} \biggl(\sum_{d|n}\chi(d)\Bigl(\frac{-t}{d}\Bigr)a\bigl(t(n/d)^2\bigr)\biggr) q^n.
\]
Then $\mathcal{S}_t(f)$ is a modular form of weight $2$ on $\Gamma_0(2N)$, with character $\chi^2$.
Assuming that $f$ is orthogonal to all cusp forms
$\sum_{n\geq 1} \psi(n)nq^{n^2}$,
where $\psi$ is an odd Dirichlet character, then $\mathcal{S}_t(f)$ is a cusp form.
\end{theorem}

One can show that if $p$ is a prime and $p\nmid 2tN$,
then $\mathcal{S}_t(f|T_{p^2})=\mathcal{S}_t(f)|T_{p}$,
where $T_m$ denotes the Hecke operator of order $m$.
In particular, if $f$ is a Hecke eigenform for $T_{p^2}$,
then $\mathcal{S}_t(f)$ is an eigenform for the same $T_{p}$,
which is the setting where Waldspurger's theorem applies.
We state it here specializing to weight $3/2$ directly.

\begin{theorem}[Waldspurger]
Suppose $f\in S_{3/2}(\Gamma_0(N),\chi)$ and $F\in S_{2}(\Gamma_0(N),\chi^2)$
verify $f|T_{p^2}=\lambda_pf$ and $F|T_{p}=\lambda_pF$ for all $p\nmid N$.
If $f(z)=\sum_{n\geq 1}a(n)q^n$, then
\[
a(n_1)^2\, L(\tfrac{1}{2},F\otimes \chi^{-1}\chi_{-n_2})\sqrt{n_2}
=
a(n_2)^2\, L(\tfrac{1}{2},F\otimes \chi^{-1}\chi_{-n_1})\sqrt{n_1}
\]
for all $n_1,n_2$ squarefree positive integers with
$\frac{n_1}{n_2}\in (\QQ_p^\times)^2$ for all $p|N$.
\end{theorem}

If the $L$-values in Waldspurger's formula are non-zero,
we can solve for one between $a(n_1)$ or $a(n_2)$ obtaining
\begin{equation}\label{0303:eq001}
a(n)^2 = A_m \sqrt{n}\, L(\tfrac{1}{2},F\otimes \chi^{-1}\chi_{-n}),
\end{equation}
where
\[
A_m = \frac{a(m)^2}{ L(\tfrac{1}{2},F\otimes \chi^{-1}\chi_{-m})\,\sqrt{m}},
\]
for all integers $n$ that are in the same squareclass as $m$ in $(\QQ_p/(\QQ_p^\times)^2)$ for all $p|N$.
In other words, \eqref{0303:eq001} shows that Fourier coefficients of weight-$3/2$
modular forms can be identified with special values of $L$-functions of modular forms in weight 2.

\subsection{Proof of Theorem \ref{thm:GRH}}\label{sec:GRH.3}
We explain in some detail the proof for $Q_{\fiveb}$ and~$Q_{23}$.
For $Q_{\fiveb}=x^2+2y^2+yz+7z^2$ we have
\[
\begin{split}
\theta_{Q_{\fiveb}}(z)
&= 1 + 2q + 2q^2 + 4q^3 + 2q^4 + 4q^6 + \cdots
\\
&=
\sum_{n\geq 0} r_{Q_{\fiveb}}(n)q^n \in M_{3/2}(\Gamma_0(220),\chi_{220}).
\end{split}
\]
The genus of $Q_{\fiveb}$ has size two and a representative
for the other class is $Q'=x^2+xy+2y^2+yz+8z^2$.
Decomposing $\theta_{Q_{\fiveb}}$ gives
\[
\begin{gathered}
E = \frac{\theta_{Q_{\fiveb}} + \theta_{Q'}}{2} = 1 + 2q + 3q^2 + 2q^3 + 4q^4 + 2q^6 + \cdots
\\
C = \theta_{Q_{\fiveb}}-E = -q^2 + 2q^3 - 2q^4 + 2q^6 + \cdots.
\end{gathered}
\]
We use the Shimura lift $\mathcal{S}_3:S_{3/2}(\Gamma_0(220),\chi_{220})\to M_2(\Gamma_0(110))$
and note that $C$ is a Hecke eigenform orthogonal to $\ker(\mathcal{S}_3)$.
Therefore, $\mathcal{S}_3(C)$ is a newform with the same system of eigenvalues as $C$,
so Waldspurger's theorem applies.
To be precise, $\mathcal{S}_3(C)=2F$, where
\[
F(z) = q + q^2 -q^4 +q^5-3q^8-3q^9+\cdots \in S_{2}(\Gamma_0(55)),
\]
the newform associated with the elliptic curve
$E: y^2+xy=x^3-x^2-4x+3$.

By Theorem \ref{theo:partial1}, we need to prove that $Q_{\fiveb}$
represents all squarefree positive integers $n$ coprime to $55$ such that $n\not\equiv 1\pmod{4}$.
For each triple
$(n_1,n_2,n_3)\in(\QQ_2^\times/(\QQ_2^\times)^2)\times(\QQ_5^\times/(\QQ_5^\times)^2)\times(\QQ_{11}^\times/(\QQ_{11}^\times)^2)$
with $n_1\not\equiv 1\pmod{4}$ and $\mathrm{ord}_5(n_2)=\mathrm{ord}_{11}(n_3)=0$, we compute a constant $d>0$ such that if $n$ is a positive squarefree integer with $n/n_1\in(\QQ_2^\times)^2$,
$n/n_2\in(\QQ_5^\times)^2$ and $n/n_3\in(\QQ_{11}^\times)^2$, we have
by \eqref{0303:eq003} and \eqref{0303:eq001}
\begin{equation}\label{1905:eq001}
r_{Q_{\fiveb}}(n) = \frac{\sqrt{220n}}{4\pi} L(1,\chi_{-220n}) \pm d n^{1/4} \sqrt{L(1/2,F\otimes\chi_{-220n})}.
\end{equation}
Note that $\chi_{-220n}$ is primitive for each squareclass we consider.
Also, rather pleasantly, we have the same $d$ in all cases, namely $d\approx 2.0757055$.
If $r_{Q_{\fiveb}}(n)=0$, we get that
\[
\frac{\sqrt{L(1/2,F\otimes\chi_{-220n})}}{L(1,\chi_{-220n})} \geq \frac{\sqrt{55}}{2\pi d} n^{1/4}.
\]
On the other hand, Chandee's theorems (which are conditional on GRH) give
\[
\frac{\sqrt{L(1/2,F\otimes\chi_{-220n})}}{L(1,\chi_{-220n})} \leq 6.2980085\, n^{0.114464}
\]
Therefore, we get a contradiction if $n\geq 50752966$.
Below this number all integers coprime to $5$ are represented,
which proves the theorem for $Q_{\fiveb}$.

Take now $Q_{23}=x^2+2xz+2y^2+3yz+5z^2$. We have
\[
\begin{split}
\theta_{Q_{23}}(z)
&= 1 + 2q + 2q^2 + 6q^3 + 8q^4 + \cdots\\
&=\sum_{n\geq 0} r_{Q_{23}}(n) q^n \in M_{3/2}(\Gamma_0(92),\chi_{92}).
\end{split}
\]
The genus of $Q_{23}$ has size 3 and the other forms are
$Q'=x^2+y^2+6z^2+xz$ and $Q''=x^2+y^2+8z^2+xy+xz$. We have
\[
\begin{gathered}
E = \frac{6}{11}\theta_{Q_{23}} + \frac{3}{11}\theta_{Q'} + \frac{2}{11}\theta_{Q''}
= 1 + \frac{36}{11}q + \frac{24}{11}q^2 + \frac{48}{11}q^3 + \frac{72}{11}q^4 + \cdots
\\
C = \theta_{Q_{23}}-E = -\frac{14}{11}q - \frac{2}{11}q^2 + \frac{18}{11}q^3 + \frac{16}{11}q^4 -\frac{4}{11}q^5 + \cdots.
\end{gathered}
\]
The Shimura lift $\mathcal{S}_3:S_{3/2}(\Gamma_0(92),\chi_{92})\to M_2(\Gamma_0(46))$
is injective and we find that
\[
C = \frac{45+3\sqrt{5}}{55} g_{+} + \frac{45-3\sqrt{5}}{55} g_{-},
\]
where $g_{+}$ and $g_{-}$ are Hecke eigenforms with Fourier expansion
\[
g_{\pm}=
\frac{-2\mp\sqrt{5}}{3}q + \frac{-1\pm\sqrt{5}}{6}q^2
+ q^3 + \frac{5\pm\sqrt{5}}{6}q^4 + \frac{-1\pm\sqrt{5}}{3}q^5 + \cdots,
\]
which are mapped by $\mathcal{S}_3$ to the newforms
\[
F^{\pm} =  q + \frac{-1 \mp\sqrt{5}}{2}q^2 + \sqrt{5}q^3 + \frac{-1\pm\sqrt{5}}{2}q^4 +(-1 \mp\sqrt{5})q^5 + \cdots \in S_2(\Gamma_0(23)).
\]
By Theorem \ref{theo:partial1}, we need to prove that $Q_{23}$
represents all squarefree positive integers $n$ coprime to $23$ such that $n\not\equiv 1\pmod{4}$.
For each pair $(n_1,n_2)\in(\QQ_2^\times/(\QQ_2^\times)^2)\times(\QQ_{23}^\times/(\QQ_{23}^\times)^2)$
with $n_1\not\equiv 1\pmod{4}$ and $\mathrm{ord}_{23}(n_2)=0$, 
we compute constants $d_1,d_2>0$ such that if $n$ is a positive squarefree integer with $n/n_1\in(\QQ_2^\times)^2$,
$n/n_2\in(\QQ_{23}^\times)^2$, we have by \eqref{0303:eq003} and \eqref{0303:eq001}
\[
\begin{split}
r_{Q_{23}}(n) = \frac{6\sqrt{92n}}{11\pi} L(1,\chi_{-92n})
&\pm d_1 n^{1/4}\sqrt{L(1/2,F^{+}\otimes\chi_{-92n})}\\
&\pm d_2 n^{1/4}\sqrt{L(1/2,F^{-}\otimes\chi_{-92n})}.
\end{split}
\]
Note that $\chi_{-92n}$ is primitive for each squareclass we consider.
Also, we find that $d_1$ and $d_2$ are the same in all cases,
namely $d_1\approx 0.57564909$ and $d_2\approx 0.25034922$.
If $r_{Q_{23}}(n)=0$, then
\[
d_1 \frac{\sqrt{L(1/2,F^{+}\otimes\chi_{-92n})}}{L(1,\chi_{-92n})}
+
d_2 \frac{\sqrt{L(1/2,F^{-}\otimes\chi_{-92n})}}{L(1,\chi_{-92n})}
\geq
\frac{6\sqrt{92}}{11\pi} n^{1/4}.
\]
Applying Chandee's theorems, we get
\[
\frac{\sqrt{L(1/2,F^{+}\otimes\chi_{-92n})}}{L(1,\chi_{-92n})}\leq 14.930625 \, n^{0.114464}
\]
and
\[
\frac{\sqrt{L(1/2,F^{-}\otimes\chi_{-92n})}}{L(1,\chi_{-92n})}\leq 9.146851 \, n^{0.114464}.
\]
It follows that if $r_{Q_{23}}(n)=0$ then $n\leq 1036286$.
In this range all integers coprime to $23$
are represented, which proves the theorem for $Q_{23}$.

The proof for the other forms is similar and we only sketch it.
The form $Q_{\fivea}=x^2+xz+2y^2+3yz+7z^2$ is in genus two and the other form
is $Q'=2x^2+3y^2+3z^2-xy+2xz+2yz$.
We use the Shimura lift $\mathcal{S}_1$ and find that the cuspidal
part $C$ of $\theta_Q\in S_{3/2}(\Gamma_0(180),\chi_{180})$
is a Hecke eigenform orthogonal to $\ker(\mathcal{S}_1)$,
so that $F=\mathcal{S}_1(C)\in S_2(\Gamma_0(45))$ has the same Hecke eigenvalues as $C$.
For each squarefree positive integer $n$ coprime to $5$,
satisfying $n\equiv 1\pmod{3}$ and $n\not\equiv 3\pmod{4}$,
there is $d>0$, depending only on the squareclass of $n$ in
$\QQ_2^\times/(\QQ_2^\times)^2\times\QQ_{3}^\times/(\QQ_{3}^\times)^2\times\QQ_{5}^\times/(\QQ_{5}^\times)^2$,
such that
\[
r_{Q_{\fivea}}(n) = \frac{\sqrt{180n}}{4\pi} L(1,\chi_{-180n}) \pm d n^{1/4} \sqrt{L(1/2,F\otimes\chi_{-180n})}
\]
We find $d\leq 1.1177145$. After correcting the Euler factors at $p=3$ to make the
character primitive we apply Chandee's theorems and deduce that if $r_{Q_{\fivea}}(n)=0$ then
\[
\frac{\sqrt{180}}{4\pi} n^{1/4} \leq 9.9318745\,d\,n^{0.114464} \leq 11.1010002 n^{0.114464}.
\]
which forces $n\leq 31846860$.
In this range we verify that all integers coprime to $5$ are represented,
concluding the proof for $Q_{\fivea}$.

The form $Q_{29}=x^2+xz+2y^2+3yz+5z^2$ is in genus three and the other forms are
$Q'=x^2+3y^2+3z^2+xy+2yz$ and $Q''=x^2+y^2+10z^2+xy+xz+yz$.
We use the Shimura lift $\mathcal{S}=\mathcal{S}_3+\mathcal{S}_7$, which is injective.
The cuspidal part $C$ of the theta series
$\theta_{Q_{29}}\in S_{3/2}(\Gamma_0(116),\chi_{116})$
is a linear combination of two Hecke eigenforms
which are mapped by $\mathcal{S}$ to two
eigenforms $F^{\pm}$ in $S_{2}(\Gamma_0(58))$.
For each squarefree positive integer $n$ with $29\nmid n$ and $n\not\equiv 3\pmod{8}$
we compute constants $d_1,d_2>0$, depending only on the squareclass of $n$
in $\QQ_2^\times/(\QQ_2^\times)^2$ and $\QQ_{29}^\times/(\QQ_{29}^\times)^2$,
such that
\[
\begin{split}
r_{Q_{29}}(n) = \frac{4\delta_2\sqrt{116n}}{7\pi} L(1,\chi_{-116n})
&\pm d_1 n^{1/4}\sqrt{L(1/2,F^{+}\otimes\chi_{-116n})}\\
&\pm d_2 n^{1/4}\sqrt{L(1/2,F^{-}\otimes\chi_{-116n})},
\end{split}
\]
where $\delta_2\geq 3/4$, $d_1\leq 0.92136892$ and $d_2\leq 0.4644352$.
We apply Chandee's theorems and deduce that if $r_{Q_{29}}(n)=0$ then
\[
\frac{3\sqrt{116}}{7\pi}n^{1/4} \leq (9.420631\, d_1 + 10.3088365\, d_2) n^{0.114464} \leq  13.4676627\, n^{0.114464},
\]
which forces $n\leq 12564515$.
In this range we verify that all integers coprime to $29$ are represented,
concluding the proof for $Q_{29}$.

The form $Q_{31}=x^2+2xz+2y^2+yz+5z^2$ is genus three and the other forms are
$Q'=x^2+2y^2+5z^2+xy+xz-yz$ and $Q''=x^2+y^2+8z^2+yz$.
The cuspidal part $C$ of the theta series
$\theta_{Q_{31}}\in S_{3/2}(\Gamma_0(124),\chi_{124})$
is a linear combination of two forms
whose images under the Shimura lift $\mathcal{S}_3$ (which is injective)
are the two Galois conjugate Hecke eigenforms $F^{\pm}$
in the two-dimensional space $S_{2}(\Gamma_0(31))$.
For each squarefree positive integer $n$ with $31\nmid n$ and $n\not\equiv 1\pmod{4}$,
we compute constants $d_1,d_2>0$, depending only on the squareclass of $n$
in $\QQ_2^\times/(\QQ_2^\times)^2$ and $\QQ_{31}^\times/(\QQ_{31}^\times)^2$,
such that
\[
\begin{split}
r_{Q_{31}}(n) = \frac{2\sqrt{124n}}{5\pi} L(1,\chi_{-124n})
&\pm d_1 n^{1/4}\sqrt{L(1/2,F^{+}\otimes\chi_{-124n})}\\
&\pm d_2 n^{1/4}\sqrt{L(1/2,F^{-}\otimes\chi_{-124n})}.
\end{split}
\]
We find that $d_1$ and $d_2$ are the same in all cases,
namely $d_1\approx 0.2542708$ and $d_2\approx 0.9177556$.
Combining this with Chandee's theorems we deduce that if $r_{Q_{31}}(n)=0$ then
\[
\frac{2\sqrt{124}}{5\pi}n^{1/4} \leq (13.6987345\, d_1 + 9.7823319\, d_2) n^{0.114464} \leq 12.460978 \, n^{0.114464},
\]
which forces $n\leq 9213750$.
In this range all integers coprime to $31$ are represented
and therefore the theorem is proved. \qed

\section{Analysis in the case of arithmetic progressions}\label{sec:stat}

There is some literature on quadratic forms representing arithmetic progressions,
see e.g.~\cite{H-K,Oh2,P-W,WuSun}.
In this section we discuss quadratic forms representing a progression,
say $m\mathbb{N}+k$, from a numerical viewpoint.
Note that the case $(m,k)=(1,0)$ goes back to Barghava and Hanke's 290-theorem,
while $(m,k)=(2,1)$ gives Rouse's 451-theorem.
Both results are computationally heavy and one reason
is the escalation process.

Say that we escalate a $(n-1)\times(n-1)$ matrix $A=(a_{ij})$ with a new truant $t$
and, for simplicity, assume $k=0$, so that $m|t$.
All possible $n\times n$ escalations of $A$ are obtained by varying the coefficients
$a_{in}$ and checking that we get a positive definite matrix.
By Cauchy-Schwarz, we approximately need to check
$|a_{in}|\leq\sqrt{mC}$, for $1\leq i\leq n-1$ and certain $C\geq 1$.
At each escalation step we need therefore to visit about
$m^{\frac{n-1}{2}}$ matrices, and completing all the steps
requires roughly $m^{\frac{n(n-1)}{4}}$ visits.
This becomes quickly unmanageable if $m$ is large.

In order to quickly get a grasp of the critical exceptions
in our arithmetic progression, we implemented a code
that selects a random branch in the escalation and stops
when it finds a matrix with no truants.
The choice of coefficients in the escalation
and the search for exceptions are made within some
prescribed height~$H$, which we took as $H=\min\{451(m-1),15000\}$.
You can find the data related to this section in \cite{publiccode}.

We tested 100 random escalations for all $m=2,\dots,100$ and $0\leq k\leq m-1$;
later we repeated with $10^3$ tosses, and finally we did $10^4$ tosses in the special
case $k=0$. The first test took overall 110h of machine time (split among eight cores),
thousand tosses required 1300h, and the last case took 150h.
Computations were done on a $11$th gen intel (R) Core(TM) i7-11657G7 @ 2.80GHz 1.69 GHz.

Let us focus on the case $k=0$ (the other cases are similar). Our numerics
produced a list $T_m$ of multiples of $m$. After dividing by $m$, we compared these
sets with the ones obtained in Theorem~\ref{theo:MAIN}.
We summarize our findings in the following observations:

\begin{itemize}[leftmargin=20pt]

   \item[\textsf{(i)}] For all $m=2,\dots,100$, we have $|T_m|\leq 57$ (occurring when $m=6$)
                       and $\max\{T_m\}\leq 383m$ (occurring when $m=12$).
                       However, we have $|T_m|\leq 36$ when $m\geq 34$,
                       suggesting that when $451(m-1)>15000$
                       the truncation at $H$ causes an `accuracy loss'.

   \item[\textsf{(ii)}] After the first $10^3$ tosses, we find very few new exceptions;
                        in other words, some elements in $T_m$ are found many times,
                        while others occur very rarely.
                        In general, we frequently obtain small truants,
                        whereas large ones appear less often.

   \item[\textsf{(iii)}] Let $S_\infty$ be the union of the sets in Table~\ref{table:Sp}
                         and let $T_\infty=\frac{T_2}{2}\cup\cdots\cup \frac{T_{100}}{100}$.
                         We have $|S_\infty|=68$ and we find $|T_\infty|=104$
                         and $|T_\infty\cap S_\infty|=48$.

\end{itemize}

Elaborating on \textsf{(i)} and \textsf{(ii)} and varying $k\neq 0$,
we also considered the quantities $T_m(k;10^2)$ and $T_m(k;10^3)$
corresponding to the truants in the arithmetic progression $k\pmod{m}$,
obtained by taking $10^2$ and $10^3$ samples, respectively.
We found that $|T_m(k;10^3)|\leq 49$ and
$6\leq |T_m(k;10^3)-T_m(k;10^2)|\leq 29$ in all cases,
indicating that $10^2$ samples are too few to get any significant indication,
and that with $10^3$ samples we are not far from the maximal cardinality we found.
On average over $1\leq k\leq m-1$, we obtained
\[
5.66 \leq \frac{1}{m-1}\sum_{k=1}^{m-1} |T_m(k;10^3)-T_m(k;10^2)| \leq 26.8
\]
and on average over both $2\leq m\leq 34$ and $1\leq k\leq m-1$ we found
\[
\frac{1}{33}\sum_{m=2}^{34} \frac{1}{m-1}\sum_{k=1}^{m-1} |T_m(k;10^3)-T_m(k;10^2)| \approx 8.65.
\]
As for the largest element, we found that $\max\{T_m(k;10^3)\}\leq 435m$
(the maximum is attained when $(m,k)=(30,5)$, when we found
the truant $13025\approx 434.16\times 30$).

It seems plausible, see Question~\ref{intro:question} in the introduction,
that there are absolute constants $C_1,C_2>0$ such that
for any arithmetic progression $k\pmod{m}$ the full critical set of exceptions $T_m(k)$
satisfies $\max\{T_m(k)\}\leq C_1m$ and $|T_m(k)|\leq C_2$, although
the evidence in favour of this (or against it) is limited.


\addresseshere

\clearpage

\appendix\section*{Appendix: quadratic forms with given truant}

For every truant $t$ appearing in $S_5\cup S_7$ with $t\notin S_1$
(see Table~\ref{table:Sp}) and every $t'\in S_5'\cup S_7'$
with $t'\notin S_1'$ (see Table~\ref{table:SpPrime}),
we list a quadratic form representing all positive integers
less than $t$ coprime to $p$, where $p\in\{5,7\}$.
For the other primes similar quadratic forms have already been found
and we refer to \cite{Bhargava,BH,DeBenedetto,Rouse}.
On the left is indicated for which primes the truant $t$ occurs.

\begin{table}[!h]
\small
\renewcommand\arraystretch{1.25}
\begin{tabular}{r|>{\centering$}p{80mm}<{$}|>{$}c<{$}} 
\hline
\phantom{xxx}$p$ &\text{Form} & \text{Truant} \\
\hline
$5$ & x^2 +2y^2 -2xz -6yz + 7z^2 & 13\\
\hline
$5$ & x^2 +2y^2 + 2yw + 7z^2 + 2zw + 14w^2 & 21\\
\hline
$7$ & x^2 + 2y^2 - 2yz + 4yw + 5z^2 + 15w^2 & 23 \\
\hline
$5$ & x^2 -4xz + 2y^2 +7z^2 & 26\\
\hline
$5$ & x^2 - 2xz + 4xw + 2y^2 + 7z^2 - 4zw + 13w^2 & 29\\
\hline
$7$ & x^2 - 2xz + 2y^2 + 5z^2 & 30\\
\hline
$7$ & x^2 + 4xw + 2y^2 - 2yz + 5z^2 + 4zw + 15w^2 & 31\\
\hline
$5$ & x^2 + xy + 2y^2 - xz - yz + 3z^2 + 19w^2 &  38\\
\hline
$5$ & x^2 + 2y^2 - 2xz - yz + 7z^2 + 3xw + 2yw - zw + 14w^2 &  39\\
$7$ & x^2 + 2y^2 - 4xz - 2yz + 5z^2 + 4xw - 8zw + 15w^2 & 39\\
\hline
$5$ & x^2 + 2y^2 - xz - 2yz + 7z^2 + 5xw + 3yw - 4zw + 14w^2 &  46\\
$7$ & x^2 + 2y^2 - 3xz - 2yz + 5z^2 + 4xw - 4zw + 30w^2 & 46\\
\hline
$5$ & x^2 + 2y^2 - 2xz - yz + 7z^2 + xw + zw + 14w^2 &  47\\
$7$ & x^2 + 2y^2 - 2yz + 5z^2 + 2yw + zw + 15w^2 & 47\\
\hline
$5$ & x^2 + 2y^2 - xz - yz + 7z^2 + 2xw + 3yw + 8zw + 21w^2 &  53\\
\hline
$7$ & x^2 + 2y^2 - 2yz + 5z^2 + 4zw + 15w^2 & 55\\
\hline
$5$ & x^2 -4xz +2y^2 +8yw +7z^2 +26w^2 & 58\\
\hline
$5$ & x^2 + 2y^2 - 2xz - 2yz + 7z^2 + 3xw + 2yw - 4zw + 13w^2 &  61\\
\hline
$5$ & x^2 + 2y^2 - xz - 2yz + 7z^2 + 5xw + yw - 3zw + 14w^2 &  62\\
$7$ & x^2 + 2y^2 - 3xz - 2yz + 5z^2 + 8yw + 30w^2 & 62\\
\hline
$5$ & x^2 + 2y^2 - 4xz + 7z^2 + 3xw + yw - 3zw + 26w^2 &  74\\
\hline
$5$ & x^2 + 2y^2 - xz - 2yz + 7z^2 + xw + yw - zw + 14w^2 & 78\\
$7$ & x^2 + 2y^2 - 4xz + 5z^2 + 4xw - 4zw + 30w^2 & 78\\
\hline
$7$ & x^2 + 2y^2 - 2xz + 5z^2 + xw + yw + 3zw + 30w^2 & 142\\
\hline
\end{tabular}
\captionsetup{size=small,width=94mm}
\caption{Quadratic forms representing all integers smaller than a given truant $t$ and coprime to the prime $p$ indicated on the left.}\label{table:forms4truants}
\end{table}

\end{document}